\documentclass[11pt,a4paper,onesided]{article}

\usepackage{ amssymb, amsmath, amsthm, graphics, xspace, enumerate}
\usepackage{graphicx}
\usepackage[a4paper,colorlinks=true,citecolor=blue,urlcolor=blue,linkcolor=blue,bookmarksopen=true,unicode=true,pdffitwindow=true]{hyperref}
 \hypersetup{pdfauthor={Nikola Sandri\'{c}}}
\hypersetup{pdftitle={On Transience of L\'evy-Type Processes}}

\newtheorem{tm}{tm}[section]
\newtheorem{theorem}[tm]{Theorem}
\newtheorem{lemma}[tm]{Lemma}
\newtheorem{corollary}[tm]{Corollary}
\newtheorem{proposition}[tm]{Proposition}

\newtheorem{example}[tm]{Example}

\newcommand {\R} {\ensuremath{\mathbb{R}}}

\newcommand {\N} {\ensuremath{\mathbb{N}}}
\newcommand {\CC} {\ensuremath{\mathbb{C}}}
\newcommand{\process}[1]{\{#1_t\}_{t\geq0}}
\newcommand{\chain}[1]{\{#1_n\}_{n\geq0}}
\numberwithin{equation}{section}
\def\be{\begin{equation}}
\def\ee{\end{equation}}
\begin{document}

 \title{On Transience of L\'evy-Type Processes}
 \author{Nikola Sandri\'{c}\\
Institut f\"ur Mathematische Stochastik\\ Fachrichtung Mathematik, Technische Universit\"at Dresden, 01062 Dresden, Germany\\
and\\
Department of Mathematics\\
         Faculty of Civil Engineering, University of Zagreb, 10000 Zagreb,
         Croatia \\
        Email: nsandric@grad.hr}

 \maketitle
\begin{center}
{
\medskip

} \end{center}

\begin{abstract}
In this paper, we study weak and strong transience of a class of Feller processes associated with pseudo-differential operators, the so-called
L\'evy-type processes. As a main result, we derive  Chung-Fuchs type
conditions (in terms of the symbol of the corresponding pseudo-differential operator) for these properties, which are sharp for L\'evy processes. Also,  as a consequence, we discuss
 the weak and strong transience  with respect to the dimension of the state space
and Pruitt indices, thus generalizing some well-known results related to elliptic diffusion and stable L\'evy processes. Finally, in the
 case when the symbol is radial (in the co-variable) we provide conditions for the weak and strong transience in terms of the L\'evy measures.
\end{abstract}

\noindent{\small \textbf{AMS 2010 Mathematics Subject Classification:} 60J25, 60J75, 60G17} \smallskip

\noindent {\small \textbf{Keywords and phrases:} L\'evy-type process, strong transience, symbol,
transience, weak transience}

%
%
% ------------------------------ INTRODUCTION ---------------------------------------
%
%

%\thispagestyle{empty}

\section{Introduction}
Let  $(\process{L},\mathbb{P})$ be a $d$-dimensional L\'evy process.
The process $\process{L}$ is said to be \emph{transient} if $$\mathbb{P}\left(\lim_{t\longrightarrow\infty}|L_t|=\infty\right)=1,$$ and \emph{recurrent} if    $$\mathbb{P}\left(\liminf_{t\longrightarrow\infty}|L_t|=0\right)=1.$$ It is well known that every L\'evy process is either transient or recurrent (see \cite[Theorem 35.3]{Sato-Book-1999}).  The  transience and recurrence of L\'evy processes can  also be (equivalently) characterized through  last exit times.
Let $O\subseteq\R^{d}$ be arbitrary. Define $L_O:=\sup\{t\geq0:L_t\in O\}$, the last exit time of $\process{L}$ from the set $O$. Clearly, if $O$ is open, then,  since $\process{L}$ has c\`adl\`ag sample paths, $L_O$ becomes a random variable. Now, we easily see that $\process{L}$ is transient if, and only if, $L_{B(0,r)}<\infty$ $\mathbb{P}$-a.s. for every $r>0$. Here, $B(x,r)$ denotes the open ball around $x\in\R^{d}$ of radius $r>0$. The previous observation suggests that the last exit times might be  suitable objects for further and deeper analysis of the transience property. Indeed,  in \cite{Sato-Watanabe-2004}
 it has been shown that for a (transient) L\'evy process $\process{L}$  and any $\kappa>0$,  $\mathbb{E}[L^{\kappa}_{B(0,r)}]$ is either finite or infinite for every $r>0$ (see also \cite{Hawkes-1977} for the case of symmetric L\'evy processes). Accordingly, a (transient) L\'evy process $\process{L}$ is said to be \emph{$\kappa$-weakly transient} if
 $$\mathbb{E}[L^{\kappa}_{B(0,r)}]=\infty\quad \textrm{for all}\ r>0,$$ and  \emph{$\kappa$-strongly transient} if $$\mathbb{E}[L^{\kappa}_{B(0,r)}]<\infty\quad \textrm{for all}\ r>0.$$
 Also, in the same reference, the authors have proved that the above conditions are equivalent to $$\int_0^{\infty}t^{\kappa}\, \mathbb{P}(L_t\in B(0,r))dt=\infty\quad \textrm{for all}\ r>0$$ and
 $$\int_0^{\infty}t^{\kappa}\, \mathbb{P}(L_t\in B(0,r))dt<\infty\quad \textrm{for all}\ r>0,$$ respectively.
However, the above characterizations  are in general not too operable, that is,  it is practically very hard to say (based just on the above characterizations) whether a given L\'evy process is $\kappa$-weakly or $\kappa$-strongly transient.  Therefore,  they
also derive certain analytic conditions for theses properties: (i) Chung-Fuchs type conditions (in terms of the corresponding characteristic exponent)  and,
 in the one-dimensional symmetric case, (ii) conditions in terms of the underlying L\'evy measure.
 Let us also remark that in their follow up paper \cite{Sato-Watanabe-2005} they have made a deep analysis of the $\kappa$-weak and $\kappa$-strong transience of (semi-)stable L\'evy processes (see also \cite{Takeuchi-1967} for the case of rotationally invariant stable L\'evy processes).
 The main goal  of this
paper is to extend the notion of the $\kappa$-weak and $\kappa$-strong transience, and derive certain more operable (analytic) conditions for these properties (which generalize the conditions presented in \cite{Sato-Watanabe-2004}), to a class of Feller processes, the so-called L\'evy-type processes.

A $d$-dimensional Feller process is a strong Markov process  $(\process{F},\{\mathbb{P}^{x}\}_{x\in\R^{d}})$
with state space $(\R^{d},\mathcal{B}(\R^{d}))$ whose associated operator semigroup $\process{P}$,
$$P_tf(x):= \mathbb{E}^{x}[f(F_t)],\quad t\geq0,\ x\in\R^{d},\  f\in C_\infty(\R^{d}),$$  forms a \emph{Feller semigroup}. This means that
\begin{itemize}
  \item [(i)] $\process{P}$ enjoys the \emph{Feller property}, that is,  $P_t(C_\infty(\R^{d}))\subseteq C_\infty(\R^{d})$ for all $t\geq0$;
  \item [(ii)] $\process{P}$ is \emph{strongly continuous}, that is, $\lim_{t\longrightarrow0}||P_tf-f||_{\infty}=0$ for all $f\in
  C_\infty(\R^{d})$.
\end{itemize}
Here, $\mathcal{B}(\R^{d})$, $C_\infty(\R^{d})$ and $||\cdot||_\infty$ denote the Borel $\sigma$-algebra on $\R^{d}$,
 space of continuous functions vanishing at infinity and  supremum norm, respectively.   Note that
every Feller semigroup $\process{P}$  can be uniquely extended to the space of bounded and Borel measurable functions
$B_b(\R^{d})$ (see \cite[Section 3]{Schilling-Positivity-1998}). For notational
simplicity, we denote this extension again by $\process{P}$. Also,
let us remark that every Feller process  has c\`adl\`ag sample paths (see  \cite[Theorems 3.4.19] {Jacob-Book-III-2005}). The \emph{infinitesimal generator}
$(\mathcal{A},\mathcal{D}_{\mathcal{A}})$ of a Feller process $\process{F}$
(or of the corresponding Feller semigroup $\process{P}$) is a linear operator
$\mathcal{A}:\mathcal{D}_{\mathcal{A}}\longrightarrow C_\infty(\R^{d})$
defined by
$$\mathcal{A}f:=
  \lim_{t\longrightarrow0}\frac{P_tf-f}{t},\quad f\in\mathcal{D}_{\mathcal{A}}:=\left\{f\in C_\infty(\R^{d}):
\lim_{t\longrightarrow0}\frac{P_t f-f}{t} \ \textrm{exists in}\
||\cdot||_\infty\right\}.
$$
Let $C_c^{\infty}(\R^{d})$ be the space of smooth functions
with compact support. Under the assumption
\begin{enumerate}
  \item [(\textbf{C1})] $C_c^{\infty}(\R^{d})\subseteq\mathcal{D}_{\mathcal{A}}$,
\end{enumerate}
 in \cite[Theorem 3.4]{Courrege-1965} it has been shown that
$\mathcal{A}|_{C_c^{\infty}(\R^{d})}$ is a \emph{pseudo-differential
operator}, that is, it can be written in the form
\begin{equation}\label{eq1.1}\mathcal{A}|_{C_c^{\infty}(\R^{d})}f(x) = -\int_{\R^{d}}q(x,\xi)e^{i\langle \xi,x\rangle}
\mathcal{F}(f)(\xi) d\xi,\end{equation}  where $\mathcal{F}(f)(\xi):=
(2\pi)^{-d} \int_{\R^{d}} e^{-i\langle\xi,x\rangle} f(x) dx$ denotes
the Fourier transform of the function $f(x)$. The function $q :
\R^{d}\times \R^{d}\longrightarrow \CC$ is called  the \emph{symbol}
of the pseudo-differential operator. It is measurable and locally
bounded in $(x,\xi)$ and continuous and negative definite as a
function of $\xi$. Hence, by \cite[Theorem 3.7.7]{Jacob-Book-I-2001}, the
function $\xi\longmapsto q(x,\xi)$ has for each $x\in\R^{d}$ the
following L\'{e}vy-Khintchine representation $$q(x,\xi) =a(x)-
i\langle \xi,b(x)\rangle + \frac{1}{2}\langle\xi,C(x)\xi\rangle -
\int_{\R^{d}}\left(e^{i\langle\xi,y\rangle}-1-i\langle\xi,y\rangle1_{B(0,1)}(y)\right)\nu(x,dy),$$
where $a(x)$ is a nonnegative Borel measurable function, $b(x)$ is
an $\R^{d}$-valued Borel measurable function,
$C(x):=(c_{ij}(x))_{1\leq i,j\leq d}$ is a symmetric non-negative
definite $d\times d$ matrix-valued Borel measurable function and $\nu(x,dy)$ is a Borel kernel on $\R^{d}\times
\mathcal{B}(\R^{d})$, called the \emph{L\'evy measure}, satisfying
$$\nu(x,\{0\})=0\quad \textrm{and} \quad \int_{\R^{d}}\min\{1,|y|^{2}\}\nu(x,dy)<\infty,\quad x\in\R^{d}.$$
The quadruple
$(a(x),b(x),C(x),\nu(x,dy))$ is called the \emph{L\'{e}vy quadruple}
of the pseudo-differential operator
$\mathcal{A}|_{C_c^{\infty}(\R^{d})}$ (or of the symbol $q(x,\xi)$).
In the sequel we also assume the following two conditions on the symbol $q(x,\xi)$:
\begin{itemize}
  \item [(\textbf{C2})] $||q(\cdot,\xi)||_\infty\leq c(1+|\xi|^{2})$ for all $\xi\in\R^{d}$ and some
  $c\geq0$;
  \item [(\textbf{C3})] $q(x,0)=0$ for all $x\in\R^{d}.$
\end{itemize}
The condition in (\textbf{C2})
 means that $q(x,\xi)$ has bounded coefficients
 (see \cite[Lemma 2.1]{Schilling-PTRF-1998}), while, according to \cite[Theorem 5.2]{Schilling-Positivity-1998}, (\textbf{C3}) (together with (\textbf{C2})) implies that $\process{F}$ is \emph{conservative}, that is, $\mathbb{P}^{x}(F_t\in\R^{d})=1$ for all $t\geq0$ and all $x\in\R^{d}.$ Note that the conservativeness is a natural assumption in the context of problems we are concerned in this paper.
Further, observe that in the case when the symbol $q(x,\xi)$ does not depend
on the variable $x\in\R^{d}$, $\process{F}$ becomes a L\'evy
process.
Moreover, every L\'evy process is uniquely and completely
characterized through its corresponding symbol (see \cite[Theorems
7.10 and 8.1]{Sato-Book-1999}). According to this, it is not hard to
check that every L\'evy process satisfies condition (\textbf{C1})  (see \cite[Theorem 31.5]{Sato-Book-1999}).
Thus, the class of processes we consider in this paper contains a subclass
 of L\'evy processes. Let us also remark here that, unlike in the case of L\'evy processes, it is not possible to associate a Feller process to every symbol (see for \cite{Bottcher-Schilling-Wang-2013} details).
  Throughout this paper, the symbol $\process{F}$ denotes a Feller
process satisfying conditions (\textbf{C1}), (\textbf{C2}) and (\textbf{C3}). Such a
process is called a \emph{L\'evy-type process}. Also,
 a L\'evy process is denoted by  $\process{L}$.
If  $\nu(x,dy)=0$ for all $x\in\R^{d}$, according to \cite[Theorem 2.44]{Bottcher-Schilling-Wang-2013}, $\process{F}$ becomes an \emph{elliptic diffusion process}. Note that this definition agrees
with the standard definition of elliptic diffusion processes (see \cite{Rogers-Williams-Book-I-2000}).
 For more
on L\'evy-type processes  we refer the readers to the monograph
\cite{Bottcher-Schilling-Wang-2013}.

Now, let us recall  the definitions of transience and recurrence of Markov processes.
Let
$(\process{M},$ $\{\mathbb{P}^{x}\}_{x\in\R^{d}})$ be a  Markov process with  c\`adl\`ag sample paths on the state space $(\R^{d},\mathcal{B}(\R^{d}))$, where $d\geq1$. The
process $\process{M}$ is called
\begin{enumerate}
  \item [(i)] \emph{irreducible} if there exists a $\sigma$-finite measure $\varphi(dy)$ on
$\mathcal{B}(\R^{d})$ such that whenever $\varphi(B)>0$ we have
$\int_0^{\infty}\mathbb{P}^{x}(M_t\in B)dt>0$ for all $x\in\R^{d}$.
 \item [(ii)] \emph{transient} if it is $\varphi$-irreducible
                       and if there exists a countable
                      covering of $\R^{d}$ with  sets
$\{B_j\}_{j\in\N}\subseteq\mathcal{B}(\R^{d})$, such that for each
$j\in\N$ there is a finite constant $c_j\geq0$ such that
$\int_0^{\infty}\mathbb{P}^{x}(M_t\in B_j)dt\leq c_j$ holds for all
$x\in\R^{d}$.
  \item [(iii)] \emph{recurrent} if it is
                      $\varphi$-irreducible and if $\varphi(B)>0$ implies $\int_{0}^{\infty}\mathbb{P}^{x}(M_t\in B)dt=\infty$ for all
                      $x\in\R^{d}$.
\end{enumerate}
It is well known that every irreducible Markov process is either transient or recurrent (see \cite[Theorem
2.3]{Tweedie-1994}). Also,
let us remark that if $\{M_t\}_{t\geq0}$ is a $\varphi$-irreducible
Markov process, then the irreducibility measure $\varphi(dy)$ can be
maximized, that is, there exists a unique ``maximal" irreducibility
measure $\psi(dy)$ such that for any measure $\bar{\varphi}(dy)$,
$\{M_t\}_{t\geq0}$ is $\bar{\varphi}$-irreducible if, and only if,
$\bar{\varphi}\ll\psi$ (see \cite[Theorem 2.1]{Tweedie-1994}).
According to this, from now on, when we refer to irreducibility
measure we actually refer to the maximal irreducibility measure.
In
the sequel, we consider  only the so-called \emph{open-set irreducible}
Markov processes, that is,
Markov processes whose maximal irreducibility measure is fully supported.
An example of such measure is the Lebesgue measure, which we denote by $\lambda(dy)$. Clearly,  a Markov process $\process{M}$ will be
$\lambda$-irreducible if $\mathbb{P}^{x}(M_t\in B)>0$ for all $t>0$ and all
 $x\in\R^{d}$ whenever $\lambda(B)>0.$ In particular, the
process $\process{M}$ will be $\lambda$-irreducible if the
transition kernel $\mathbb{P}^{x}(M_t\in dy)$, $t>0$, $x\in\R^{d}$,  possesses a strictly positive transition density
function.
Let us remark here that the $\lambda$-irreducibility of L\'evy-type processes is a very well-studied topic in the literature. We refer the readers to  \cite{Sheu-1991} and \cite{Stramer-Tweedie-1997} for the case of elliptic diffusion processes, to \cite{Kolokoltsov-2000}  for the case of a class of pure jump L\'evy-type processes (the so-called stable-like processes),
to   \cite{Bass-Cranston-1986},  \cite{Ishikawa-2001}, \cite{Knopova-Kulik-2014}, \cite{Kulik-2007}, \cite{Masuda-2007, Masuda-Erratum-2009}, \cite{Kwoon-Lee-1999} and \cite{Picard-1996, Picard-Erratum-2010} for the case of a class of L\'evy-type processes obtained as a solution of certain jump-type stochastic differential equations and
\cite{Knopova-Schilling-2012} and \cite{Knopova-Schilling-2013}   for the case of general L\'evy-type processes.

As the first main result of this paper we prove that  an open-set irreducible L\'evy-type process $\process{F}$ is transient if, and only if, there exist $x\in\R^{d}$ and an open bounded set $O\subseteq\R^{d}$, such that $L_O<\infty$ $\mathbb{P}^{x}$-a.s.  Accordingly, in the case of transience, for a Borel function $f:[0,\infty)\longrightarrow[0,\infty)$,  we say that $\process{F}$ is $f$-strongly transient if $\mathbb{E}^{x}[f(L_O)]<\infty$ for all $x\in\R^{d}$ and all open bounded sets $O\subseteq\R^{d}$. Otherwise, we say that $\process{F}$ is $f$-weakly transient. Next, in the case when the function
$f(t)$ is continuously differentiable and non-decreasing, we derive Chung-Fuchs type conditions  for the $f$-weak and $f$-strong transience of  $\process{F}$.  Specially,  if $\process{F}$ is a  L\'evy process,  these conditions reduce to the
Chung-Fuchs type conditions obtained in \cite{Sato-Watanabe-2004}.  This
shows that our conditions
are sharp for L\'evy processes.
In the special case when $f(t)=t^{\kappa}$ for some $\kappa>0$, we use the terminology
$\kappa$-weak and $\kappa$-strong transience introduced in \cite{Sato-Watanabe-2004}.
In this context,
we first discuss the $\kappa$-weak and $\kappa$-strong transience of $\process{F}$ with respect to the dimension of the state space and Pruitt indices and, in particular, the $\kappa$-weak and $\kappa$-strong transience of elliptic diffusion and stable-like processes, thus generalizing the results obtained in \cite{Li-Liu-2006} and \cite{Takeuchi-1967}. Finally, in the case when the symbol of $\process{F}$ is radial  in the co-variable we  provide a series of conditions for the $\kappa$-weak and $\kappa$-strong transience          in terms of the corresponding L\'evy measure.

 There is a vast  literature and great efforts have been made to study the recurrence (and certain related properties such as null and positive recurrence, ergodicity and strong, subexponential and exponential ergodicity) of general and certain special subclasses of Markov processes. On the other hand, the transience property of Markov processes remained largely unexplored.
The notion of  weak and strong transience, as a ``measure" of degree of transience, originated (and recognized to be crucial) from studying the asymptotic behavior as $t\longrightarrow\infty$ of the quantity $\int_{\R^{d}}\mathbb{P}^{x}(T_B\leq t)dx$ (the volume of a so-called L\'evy sausage), where $B\in\mathcal{B}(\R^{d})$ is bounded and $T_B:=\inf\{t\geq0: L_t\in B\}$ is the first hitting time of  $B$ by a transient L\'evy process $\process{L}$  (see \cite{Port-1990} and the references therein).
 Except for L\'evy processes, whose weak and strong transience have been studied extensively in \cite{Sato-Watanabe-2004} (see also \cite{Hawkes-1977} for the case of symmetric L\'evy processes and  \cite{Sato-Watanabe-2005} and \cite{Takeuchi-1967} for the case of (semi-)stable L\'evy processes),  weak and strong transience of only few more special classes of Markov
processes have  been considered
 in the literature. More precisely, in \cite{Li-Liu-2006} the authors have discussed   weak and strong transience
of elliptic diffusion processes. In \cite{Yamamuro-1998}   weak and strong transience of Ornstein-Uhlenbeck type process (which are a subclass of L\'evy-type processes with unbounded coefficients (see \cite{Sato-Yamazato-1984})) have been investigated and in \cite{Dawson-Gorostiza-Wakolbinger-2005} the authors have considered  weak and strong transience of a class of L\'evy processes on metric Abelian groups. Finally, in \cite{Wu-Zhang-Liu-2006}  strong transience
of operator-self-similar Markov processes has been studied.

This paper is organized as follows.
 First, in Section
\ref{s2}, for a Borel function $f:[0,\infty)\longrightarrow[0,\infty)$, we introduce the notion of   $f$-weak and $f$-strong  transience of L\'evy-type processes and   derive Chung-Fuchs type conditions for  these properties. In Section \ref{s3}, we restrict to  the $\kappa$-weak and $\kappa$-strong  transience (that is, the case when $f(t)=t^{\kappa}$ for some $\kappa>0$) and discuss these properties with respect to the dimension of the state space and Pruitt indices.
 Finally, in Section \ref{s4}, we consider   $\kappa$-weak and $\kappa$-strong  transience of L\'evy-type processes with radial symbol (in the co-variable) and  provide conditions for these properties in terms of the L\'evy measures.

\section{Weak and Strong Transience of L\'evy-Type Processes}\label{s2}
In this section, we  introduce the notion and derive Chung-Fuchs type conditions for weak and strong transience  of L\'evy-type processes. We start with the following characterization of the transience property of L\'evy-type processes.
\begin{theorem}\label{tm3.1} Let $\process{F}$ be a $d$-dimensional open-set irreducible  L\'evy-type process. Then,  the following properties are equivalent:
\begin{itemize}
  \item [(i)] $\process{F}$ is transient;
  \item [(ii)] there exists (for all)
$x\in\R^{d}$ such that
$$\mathbb{P}^{x}\left(\lim_{t\longrightarrow\infty}|F_t|=\infty\right)=1;$$
\item[(iii)] there exist (for all) $x\in\R^{d}$ and  an (all) open bounded set $O\subseteq\R^{d}$, such that $L_{O}<\infty$ $\mathbb{P}^{x}$-a.s.
\end{itemize}
\end{theorem}
\begin{proof}\begin{description}
               \item[(i)$\Leftrightarrow$(ii)]  Let $\chain{\bar{F}}$ be a $d$-dimensional Markov chain generated by a transition function of the form $$\bar{p}(x,dy):=\int_0^{\infty}e^{-t}\mathbb{P}^{x}(F_t\in dy)dt.$$
Clearly, $\chain{\bar{F}}$ is also open-set irreducible, and hence, due to \cite[Theorem 2.3]{Tweedie-1994}, it is either transient or recurrent. Now, according to
\cite[Theorem 2.1]{Tuominen-Tweedie-1979} (which states that $\chain{\bar{F}}$ is transient if, and only if, $\process{F}$ is transient) and \cite[Corollary 3.4 and Theorem 5.2]{Schilling-Positivity-1998}, \cite[Theorem 7.1]{Tweedie-1994} and \cite[Proposition 3.2]{Meyn-Tweedie-AdvAP-II-1993} (which state that the property in (ii) holds for $\chain{\bar{F}}$ if, and only if, it holds for $\process{F}$), the assertion will follow if we prove that $\chain{\bar{F}}$ is transient if, and only if, there exists (for all)
$x\in\R^{d}$ such that
$$\bar{\mathbb{P}}^{x}\left(\lim_{n\longrightarrow\infty}|\bar{F}_n|=\infty\right)=1,$$ where $\{\bar{\mathbb{P}}^{x}\}_{x\in\R^{d}}$ is induced by $\bar{p}(x,dy)$, $x\in\R^{d}$.
The necessity is a direct consequence of  \cite[Theorem  9.2.2 (i)]{Meyn-Tweedie-Book-2009} (together with \cite[Corollary 3.4 and Theorem 5.2]{Schilling-Positivity-1998} and \cite[Proposition 6.1.1 and Theorem 6.2.9]{Meyn-Tweedie-Book-2009}). On the other hand, the sufficiency is  a consequence of (a straightforward generalization of) \cite[Proposition 5.3]{Sandric-Bernoulli-2013} and \cite[Theorem 9.2.2 (ii)]{Meyn-Tweedie-Book-2009}.
               \item[(i)$\Leftrightarrow$(iii)] First, recall that, due to the c\`adl\`ag property of the sample paths of $\process{F}$, $L_O$  is a well-defined random variable.
               Now, the necessity follows directly from (ii). In order to conclude the sufficiency, observe that $\chain{\bar{F}}$ (the chain defined in the first part of the proof) can be represented as follows. Let $\{J_n\}_{n\geq1}$ be a sequence of i.i.d. random variables with exponential distribution (with parameter one) defined on a probability space $(\bar{\Omega},\bar{\mathcal{F}},\bar{\mathbb{P}})$. Further, set $S_0:=0$ and $S_n:=J_1+\cdots+J_n$, $n\geq1$. Then, $\{F_{S_n}\}_{n\geq0}$ is a Markov chain defined on $(\Omega\times\bar{\Omega},\mathcal{F}\times\bar{\mathcal{F}},\{\mathbb{P}^{x}\times\bar{\mathbb{P}}\}_{x\in\R^{d}})$ having the same marginal distributions  as $\chain{\bar{F}}$.
               Clearly, according to the strong law of large numbers, $\lim_{n\longrightarrow\infty}S_n=\infty$ $\mathbb{P}$-a.s. Now, the assertion follows as a consequence of the transience and recurrence dichotomy of irreducible Markov chains and  (a multidimensional version of) \cite[Proposition 2.4]{Sandric-JOTP-2014}.
             \end{description}

\end{proof}
Accordingly,  in the sequel we always assume
 \begin{itemize}
  \item [(\textbf{C4})] $\process{F}$ is a  $d$-dimensional, open-set irreducible and transient L\'evy-type process.
\end{itemize}
 Now, we introduce the notion of weak and strong transience of L\'evy-type processes.
Let $f:[0,\infty)\longrightarrow[0,\infty)$ be an arbitrary Borel measurable function.
We say that $\process{F}$ is \emph{$f$-strongly transient} if $\mathbb{E}^{x}[f(L_{O})]<\infty$ for all $x\in\R^{d}$ and all open bounded sets $O\subseteq\R^{d}$. Otherwise, we say that $\process{F}$ is  \emph{$f$-weakly transient}.
In the following lemma we slightly generalize the result from \cite[Lemma 2.2]{Sato-Watanabe-2004}.

\begin{lemma}\label{lm3.3} Let $\process{F}$ be a L\'evy-type process satisfying condition (\textbf{C4}) and let   $f:[0,\infty)\longrightarrow[0,\infty)$ be a non-decreasing and continuously differentiable function. Then, for any $x\in\R^{d}$ and any open bounded sets $O_1,O_2\subseteq\R^{d}$, $\bar{O}_1\subseteq O_2$, there is a constant $c\geq 1$ such that  $$c^{-1}\int_0^{\infty}f(t)\mathbb{P}^{x}(F_t\in \bar{O}_1)dt\leq\mathbb{E}^{x}[f(L_{\bar{O}_1})]\leq c\int_0^{\infty}f(t)\mathbb{P}^{x}(F_t\in \bar{O}_2)dt.$$ Here, $\bar{B}$ denotes the closure of the set $B\subseteq\R^{d}$.\end{lemma}
\begin{proof}
Clearly, without loss of generality, we can assume that $f(0)=0$. Now, for any $x\in\R^{d}$ and any open set $O\subseteq\R^{d}$,
\begin{align*}\mathbb{E}^{x}[f(L_{\bar{O}})]&=\int_{[0,\infty)}f(t)d\mathbb{P}^{x}(L_{\bar{O}}\leq t)\\
&=\int_{[0,\infty)}\int_{[0,t)}f'(s)dsd\mathbb{P}^{x}(L_{\bar{O}}\leq t)\\&=\int_0^{\infty}f'(t)\mathbb{P}^{x}(L_{\bar{O}}> t)dt\\&=
\int_0^{\infty}f'(t)\mathbb{P}^{x}(T_{\bar{O}}\circ\theta_t<\infty)dt\\&=
\int_0^{\infty}f'(t)\,\mathbb{E}^{x}[\mathbb{P}^{F_t}(T_{\bar{O}}<\infty)]dt,\end{align*}
where  in the fourth step we used the fact $\{L_{\bar{O}}>t\}=\{T_{\bar{O}}\circ\theta_t<\infty\}$. Recall that $T_B$ denotes the first hitting time of the set $B\subseteq\R^{d}$ by $\process{F}$.
Further,  \cite[Lemma 2.2]{Yamamuro-1998} states that if \begin{equation}\label{eq4}\inf_{x\in K}\mathbb{E}^{x}\left[e^{- L_K}\right]>0\quad\textrm{ for every compact set}\ K\subseteq\R^{d},\end{equation} then for any $x\in\R^{d}$ and any open bounded sets $O_1,O_2\subseteq\R^{d}$, $\bar{O}_1\subseteq O_2$,   there is a constant $c\geq 1$ such that
$$c^{-1}\int_0^{\infty}\int_0^{\infty}\mathbb{E}^{x}[\mathbb{P}^{F_t}(F_s\in \bar{O}_1)]ds f'(t)dt\leq\mathbb{E}^{x}[f(L_{\bar{O}_1})]\leq c\int_0^{\infty}\int_0^{\infty}\mathbb{E}^{x}[\mathbb{P}^{F_t}(F_s\in \bar{O}_2)]ds f'(t)dt.$$
On the other hand, for any $x\in\R^{d}$ and any open set $O\subseteq\R^{d}$,
$$\int_0^{\infty}\int_0^{\infty}\mathbb{E}^{x}[\mathbb{P}^{F_t}(F_s\in \bar{O})]dsf'(t)dt=\int_0^{\infty}\mathbb{E}^{x}\left[\int_t^{\infty}1_{\bar{O}}(F_s)ds\right]f'(t)dt=\int_0^{\infty}f(t)\mathbb{P}^{x}(F_t\in \bar{O})dt,$$ which proves the claim.  Finally, let us prove the relation in \eqref{eq4}. Let $K\subseteq\R^{d}$ be an arbitrary compact set. First, observe that,  by the Markov property,
$$\mathbb{E}^{x}[e^{-L_K}]\geq e^{-t}\int_{\R^{d}}\mathbb{E}^{y}[e^{-L_K}]\mathbb{P}^{x}(F_t\in dy),\quad x\in\R^{d},\ t\geq0.$$ Further, according to \cite[Corollary 3.4 and Theorem 5.2]{Schilling-Positivity-1998} and \cite[Theorem 7.1]{Tweedie-1994}, there exist a probability measure $a(dt)$ on $\mathcal{B}([0,\infty))$ and a Borel kernel $T:\R^{d}\times\mathcal{B}(\R^{d})\longrightarrow[0,1]$, such that
\begin{itemize}
\item [(i)] $T(x,\R^{d})>0$ for all $x\in\R^{d};$
  \item [(ii)] $\displaystyle\int_{[0,\infty)}\mathbb{P}^{x}(F_t\in B)a(dt)\geq T(x,B)$ for all $x\in\R^{d}$ and all $B\in\mathcal{B}(\R^{d});$
  \item [(iii)]  $\displaystyle\liminf_{y\longrightarrow x}\int_{\R^{d}}f(z)T(y,dz)\geq\int_{\R^{d}}f(z)T(x,dz)$  for all $x\in\R^{d}$ and all $f\in B_b(\R^{d}),$ $f\geq0.$
\end{itemize} Now, fix $t_0>0$. Then, from the previous relation we have
\begin{align*}&\mathbb{E}^{x}[e^{-L_K}]\int_{[0,t_0]}a(dt)\\&\geq e^{-t_0}\int_{[0,t_0]}\int_{\R^{d}}\mathbb{E}^{y}[e^{-L_K}]\mathbb{P}^{x}(F_t\in dy)a(dt)\\
&=e^{-t_0}\int_{[0,\infty)}\int_{\R^{d}}\mathbb{E}^{y}[e^{-L_K}]\mathbb{P}^{x}(F_t\in dy)a(dt)-e^{-t_0}\int_{(t_0,\infty)}\int_{\R^{d}}\mathbb{E}^{y}[e^{-L_K}]\mathbb{P}^{x}(F_t\in dy)a(dt)\\
&\geq e^{-t_0}\int_{\R^{d}}\mathbb{E}^{y}[e^{-L_K}]T(x,dy)-e^{-t_0}\int_{(t_0,\infty)}a(dt).\end{align*}
Hence, if \eqref{eq4} did not hold for $K$, then we would have
$$\int_{(t_0,\infty)}a(dt)\geq\inf_{x\in K}\int_{\R^{d}}\mathbb{E}^{y}[e^{-L_K}]T(x,dy),$$ that is, by letting $t_0\longrightarrow\infty$,
$$ \inf_{x\in K}\int_{\R^{d}}\mathbb{E}^{y}[e^{-L_K}]T(x,dy)=0.$$ In particular, due to the compactness of $K$, there would exist a sequence $\{x_n\}_{n\geq1}\subseteq K$, $\lim_{n\longrightarrow\infty}x_n=x_0\in K$, such that
$$0=\lim_{n\longrightarrow\infty}\int_{\R^{d}}\mathbb{E}^{y}[e^{-L_K}]T(x_n,dy)\geq\liminf_{x\longrightarrow x_0}\int_{\R^{d}}\mathbb{E}^{y}[e^{-L_K}]T(x,dy)\geq\int_{\R^{d}}\mathbb{E}^{y}[e^{-L_K}]T(x_0,dy).$$ However,
this cannot be the case since $T(x_0,\R^{d})>0$ and, due to the transience of $\process{F}$, $\mathbb{E}^{x}[e^{-L_K}]>0$ for all $x\in\R^{d}.$
\end{proof}
Let us remark that if the semigroup $\process{P}$ of $\process{F}$ satisfies the \emph{strong Feller property}, that is, $P_t(B_b(\R^{d}))\subseteq
C(\R^{d})$ for all $t > 0$, then Theorem \ref{tm3.1} and the relation in \eqref{eq4} follow directly from (a multidimensional version of)  \cite[Proposition 2.4]{Sandric-JOTP-2014} and \cite[Remark 2.3]{Yamamuro-1998}, respectively. For sufficient conditions for a Feller semigroup to be a strong Feller semigroup see \cite{Schilling-Wang-2012} and \cite{Schilling-Wang-2013}.
In the following two theorems   we derive Chung-Fuchs type conditions for the  $f$-weak and $f$-strong transience.
\begin{theorem}\label{tm3.4}Let  $f:[0,\infty)\longrightarrow[0,\infty)$ be a  non-decreasing and continuously differentiable function and let  $\process{F}$ be a L\'evy-type process  with  infinitesimal generator $(\mathcal{A},\mathcal{D}_{\mathcal{A}})$ and symbol
  $q(x,\xi)$,  satisfying condition (\textbf{C4}). Assume
  \begin{equation}\label{eq3}\liminf_{\alpha\longrightarrow0}\int_{\R^{d}}\left(\int_{\frac{\ln 2}{4\sup_{x\in\R^{d}}|q(x,\xi)|}}^{\infty}e^{-\alpha t}f(t){\rm Re}\,\mathbb{E}^{0}[\it{e}^{\it{i}\langle\xi, F_t\rangle}]dt\right)\frac{\sin^{2}\left(\frac{a\xi_1}{2}\right)}{
\xi_1^{2}}\cdots\frac{\sin^{2}\left(\frac{a\xi_d}{2}\right)}{
\xi_d^{2}}d\xi>-\infty\end{equation} for all $a>0$ small enough.
Then,  $\process{F}$ is $f$-weakly transient if
 \begin{equation}\label{eq3.1}\int_{B(0,r)}\int_{0}^{\frac{\ln 2}{4\sup_{x\in\R^{d}}|q(x,\xi)|}}f(t)dt\,d\xi=\infty\quad\textrm{for some}\ r>0.\end{equation}

\end{theorem}
\begin{proof} We follow the proof of \cite[Theorem 1.2]{Sandric-TAMS-2014} and prove that, under the above assumptions,
$\mathbb{E}^{0}[L_{O_0}]=\infty$
for every open neighborhood $O_0\subseteq\R^{d}$ of the origin.
According to Lemma \ref{lm3.3}, it suffices to prove that $$\int_0^{\infty}f(t)\mathbb{P}^{0}(F_t\in O_0)dt=\infty$$ for every open neighborhood $O_0\subseteq\R^{d}$ of the origin.
Let $a>0$ be arbitrary. By the monotone convergence theorem, we have
\begin{align*}\int_0^{\infty}f(t)\mathbb{P}^{0}(F_t\in (-a,a)^{d})dt&=\lim_{\alpha\longrightarrow0}\mathbb{E}^{0}\left[\int_0^{\infty}e^{-\alpha
t}f(t)1_{\left\{F_t\in
\left(-a,a\right)^{d}\right\}}dt\right]\\&=\lim_{\alpha\longrightarrow0}\int_0^{\infty}\int_{\R^{d}}e^{-\alpha
t}f(t)1_{ \left(-a,a\right)^{d}}(y)\mathbb{P}^{0}(F_t\in
dy)dt,\end{align*} where
$\left(-a,a\right)^{d}:=\left(-a,a\right)\times\ldots\times\left(-a,a\right)$.
Next, let
$$g(u):=\left(1-\frac{|u|}{a}\right)1_{\left(-a,a\right)}(u),\quad u\in\R,$$ and
$$h(y):=g(y_1)\cdots g(y_d),\quad y=(y_1,\ldots,y_d)\in\R^{d}.$$
Clearly,   $$ 1_{ \left(-a,a\right)^{d}}(y)\geq h(y),\quad y\in\R^{d}.$$ According to this,
$$\int_0^{\infty}f(t)\mathbb{P}^{0}(F_t\in (-a,a)^{d})dt\geq\liminf_{\alpha\longrightarrow0}\int_0^{\infty}\int_{\R^{d}}e^{-\alpha
t}f(t)h(y)\mathbb{P}^{0}(F_t\in dy)dt.$$ Further, observe that
$$g(u)=\frac{1}{\sqrt{a}}1_{\left(-\frac{a}{2},\frac{a}{2}\right)}\ast\frac{1}{\sqrt{a}}1_{\left(-\frac{a}{2},\frac{a}{2}\right)}(u),\quad u\in\R,$$
where $\ast$ denotes the standard convolution operator. Hence, since
$$\mathcal{F}\left(\frac{1}{\sqrt{a}}1_{\left(-\frac{a}{2},\frac{a}{2}\right)}\right)(\xi)=\frac{\sin\left(\frac{a\xi}{2}\right)}{\sqrt{a}\pi
\xi},$$ we have
$$\mathcal{F}(h)(\xi)=\frac{\sin^{2}\left(\frac{a\xi_1}{2}\right)}{a\pi^{2}
\xi_1^{2}}\cdots\frac{\sin^{2}\left(\frac{a\xi_d}{2}\right)}{a\pi^{2}
\xi_d^{2}}.$$ Denote by  $\Phi_t(x,\xi):=\mathbb{E}^{x}[\it{e}^{\it{i}\langle\xi, F_t-x\rangle}]$ for $t\geq0$ and $x,\xi\in\R^{d}.$ The above facts yield
\begin{align*}\int_0^{\infty}f(t)\mathbb{P}^{0}(F_t\in (-a,a)^{d})dt&\geq\liminf_{\alpha
\longrightarrow0}\int_0^{\infty}\int_{\R^{d}}\int_{\R^{d}}e^{-\alpha
t}f(t)e^{i\langle\xi,y\rangle}\mathcal{F}(h)(\xi)d\xi\,
\mathbb{P}^{0}(F_t\in
dy)dt\nonumber\\&=\liminf_{\alpha\longrightarrow0}\int_0^{\infty}\int_{\R^{d}}e^{-\alpha
t}f(t)\Phi_t(0,\xi)\mathcal{F}(h)(\xi)d\xi\,
dt\nonumber\\&=\liminf_{\alpha\longrightarrow0}\int_0^{\infty}\int_{\R^{d}}e^{-\alpha
t}f(t){\rm Re}\,\Phi_{t}(\rm{0},\xi)\mathcal{F}(\it{h})(\xi)d\xi\,
d\it{t}.\end{align*}
Now,  (the proof of) \cite[Lemma 2.2]{Sandric-TAMS-2014}  states that
$${\rm Re}\,\Phi_{t}(0,\xi)\geq\exp\left[-4t\sup_{x\in\R^{d}}|q(x,\xi)|\right],\quad t\in\left[0,\frac{\ln 2}{4\sup_{x\in\R^{d}}|q(x,\xi)|}\right].$$
Hence,
\begin{align*}&\int_0^{\infty}f(t)\mathbb{P}^{0}(F_t\in (-a,a)^{d})dt\\&\geq\liminf_{\alpha\longrightarrow0}\int_{\R^{d}}\Bigg(\int_0^{t_0(\xi)}e^{-\alpha
t}f(t){\rm Re}\,\Phi_{t}(0,\xi)dt+\int_{t_{0}(\xi)}^{\infty}e^{-\alpha t}f(t){\rm Re}\,\Phi_{t}(0,\xi)dt\Bigg)\mathcal{F}(h)(\xi)d\xi
\\&\geq
\liminf_{\alpha\longrightarrow0}\int_{\R^{d}}\left(\int_{0}^{t_0(\xi)}f(t)dt\right)\exp\left[-t_0(\xi)(\alpha+4\sup_{x\in\R^{d}}|q(x,\xi)|)\right]\mathcal{F}(h)(\xi)d\xi
\\&\ \ \ +\liminf_{\alpha\longrightarrow0}\int_{\R^{d}}\int_{t_{0}(\xi)}^{\infty}e^{-\alpha t}f(t){\rm Re}\,\Phi_{t}(0,\xi)dt\mathcal{F}(h)(\xi)d\xi,\end{align*}
where
$t_0(\xi):=\ln2/4\sup_{x\in\R^{d}}|q(x,\xi)|$. Consequently, by  the monotone convergence theorem, we
have
\begin{align*}&\int_0^{\infty}f(t)\mathbb{P}^{0}(F_t\in (-a,a)^{d})dt\\&\geq\frac{1}{2}\int_{\R^{d}}\left(\int_{0}^{t_0(\xi)}f(t)dt\right)\mathcal{F}(h)(\xi)d\xi+\liminf_{\alpha\longrightarrow0}\int_{\R^{d}}\int_{t_{0}(\xi)}^{\infty}e^{-\alpha t}f(t){\rm Re}\,\Phi_{t}(0,\xi)dt\mathcal{F}(h)(\xi)d\xi.\end{align*}
Finally, let $r>0$ be such that
$$\int_{B(0,r)}\int_{0}^{t_0(\xi)}f(t)dt\,d\xi=\infty.$$ Then, since
$$\lim_{a\longrightarrow0}\frac{(2\pi)^{2d}}{a^{d}}\mathcal{F}(h)(\xi)=1,$$ for any $c\in(0,1)$, all $\xi\in B(0,r)$ and all $a>0$ small enough
we have \begin{align*}&\int_0^{\infty}f(t)\mathbb{P}^{0}(F_t\in (-a,a)^{d})dt\\&\geq\frac{ca^{d}}{2(2\pi)^{2d}}\int_{B(0,r)}\int_{0}^{t_0(\xi)}f(t)dt\,d\xi+\liminf_{\alpha\longrightarrow0}\int_{\R^{d}}\int_{t_{0}(\xi)}^{\infty}e^{-\alpha t}f(t){\rm Re}\,\Phi_{t}(0,\xi)dt\mathcal{F}(h)(\xi)d\xi,\end{align*} which together with \eqref{eq3} concludes the proof.
\end{proof}

\begin{proposition}\label{p3.2}Let $\process{F}$ be a L\'evy-type  process  with infinitesimal generator $(\mathcal{A},\mathcal{D}_{\mathcal{A}})$ and symbol
  $q(x,\xi)$. Assume that
 \begin{itemize}
   \item [(i)] $C_c^{\infty}(\R^{d})$
is an operator core for
$(\mathcal{A},\mathcal{D}_{\mathcal{A}})$, that is,
$\mathcal{A}$ is the unique extension of $\mathcal{A}|_{C_c^{\infty}(\R^{d})}$
 on $\mathcal{D}_{\mathcal{A}}$;
   \item [(ii)]  $q(x,\xi)=q(-x,-\xi)$ for all $x,\xi\in\R^{d}$.
 \end{itemize}
Then, $\Phi_{t}(0,\xi)={\rm Re}\,\Phi_{t}(0,\xi)$ for all $t\geq0$ and all $\xi\in\R^{d}.$
 \end{proposition}
\begin{proof}
Define $\bar{F}_t:=-F_t$ for $t\geq0$.
It is straightforward to check that the process $(\process{\bar{F}},\{\bar{\mathbb{P}}^{x}\}_{x\in\R^{d}})$, where  $\bar{\mathbb{P}}^{x}(\bar{F}_t\in dy):=\mathbb{P}^{-x}(-F_t\in dy)$ for $t\geq0$ and $x\in\R^{d},$ is a Feller process
 with
symbol $\bar{q}(x,\xi):=q(-x,-\xi)$ for $x,\xi\in\R^{d}$.
Furthermore, since $C_c^{\infty}(\R^{d})$
is an operator core for
$(\mathcal{A},\mathcal{D}_{\mathcal{A}})$ and,
by assumption,  $\bar{q}(x,\xi)=q(x,\xi)$ for all $x,\xi\in\R^{d}$, \cite[Proposition 4.1.7]{Ethier-Kurtz-Book-1986} and \cite[Proposition 4.6]{Schilling-Wang-2013} state that
$(\process{F},\{\mathbb{P}^{x}\}_{x\in\R^{d}})$ and $(\process{\bar{F}},\{\bar{\mathbb{P}}^{x}\}_{x\in\R^{d}})$ have the same finite-dimensional distributions. In particular, for any $t\geq0$ and any $x,\xi\in\R^{d}$, $$\Phi_t(x,\xi)=\mathbb{E}^{x}\left[e^{i\langle F_t-x,\xi\rangle}\right]=
\bar{\mathbb{E}}^{x}\left[e^{i\langle \bar{F}_t-x,\xi\rangle}\right]=\mathbb{E}^{-x}\left[e^{i\langle -F_t+x,\xi\rangle}\right]=\mathbb{E}^{-x}\left[e^{i\langle F_t-x,-\xi\rangle}\right]=\Phi_t(-x,-\xi),$$ which proves the assertion.
\end{proof}
Observe that  L\'evy processes and elliptic diffusion processes always satisfy (i). Also, many  jump L\'evy-type processes, like stable-like processes (see Example \ref{e4.5}), L\'evy-type processes generated by finite L\'evy measures (see Example \ref{e2}) and L\'evy-type processes with sufficiently regular symbol (see \cite[Theorem 3.23]{Bottcher-Schilling-Wang-2013}),  have this property. The condition in (ii) implies (actually it is equivalent) that the corresponding L\'evy triplet $(b(x),C(x),\nu(x,dy))$ satisfies: $b(x)=-b(-x)$,  $C(x)=C(-x)$ and $\nu(x,dy)=\nu(-x,-dy)$ for all $x\in\R^{d}$.

 The condition in \eqref{eq3} is a crucial assumption in Theorem \ref{tm3.4}. Clearly, it trivially holds true in the case of a symmetric L\'evy process. Moreover,  in that case the left hand side in  \eqref{eq3} is nonnegative. We conjecture that the same property should hold in the case of a L\'evy-type process satisfying the assumptions from Proposition \ref{p3.2}. Directly from Bochner's theorem (see \cite[Theorem 3.5.7]{Jacob-Book-I-2001}) we get the following.
\begin{proposition}\label{p3}Let $\process{F}$ be a L\'evy-type  process which admits a transition density function $p(t,x,y)$, $t>0$, $x,y\in\R^{d}$, satisfying:
 \begin{itemize}
  \item [(i)] $\displaystyle\int_{\R^{d}}|\Phi_t(0,\xi)|d\xi<\infty$ for all $t>0$;
   \item [(ii)] the function $y\longmapsto p(t,0,y)$ is continuous for all $t>0$;
   \item [(iii)]  the function $y\longmapsto p(t,0,y)$ is positive definite for all $t>0$, that is, for all $n\in\N$, all $y_1,\ldots,y_n\in\R^{d},$ all $c_1,\ldots,c_n\in\CC$ and all $t>0$, $$\sum_{i=1}^{n}\sum_{j=1}^{n}c_i\bar{c}_jp(t,0,y_i-y_j)\geq0.$$
  \end{itemize}
Then, $\Phi_{t}(0,\xi)\geq0$ for all $t\geq0$ and all $\xi\in\R^{d}.$
 \end{proposition}
 Let us remark that the existence of a transition density function and (i)   follow if $$\int_{\R^{d}}\exp\left[-t\inf_{x\in\R^{d}}{\rm Re}\,q(x,\xi)\right]d\xi<\infty,\quad t>0,$$ and \eqref{eq0} hold (see \cite[Theorem 2.7]{Schilling-Wang-2013}). Moreover, under the same assumptions, (ii) follows as a direct consequence of the dominated convergence theorem. On the other hand, (iii)  holds if the functions $y\longmapsto p(t,0,y)$ is symmetric (which is the case under the assumptions of Proposition \ref{p3.2}) and $p(t,0,0)=\sup_{y\in\R^{d}}p(t,0,y)$ for all $t>0$.

\begin{theorem}\label{tm3.5}Let  $f:[0,\infty)\longrightarrow[0,\infty)$ be a  non-decreasing and continuously differentiable function and let  $\process{F}$ be a L\'evy-type process  with
 symbol $q(x,\xi)$, satisfying condition (\textbf{C4}). Then, $\process{F}$ is $f$-strongly transient if
 there exists  a constant $0\leq c<1$ such that  \begin{equation}\label{eq0}\sup_{x\in\R^{d}}|{\rm Im}\, q(x,\xi)|\leq c\inf_{x\in\R^{d}}{\rm Re}\,q(x,\xi),\quad \xi\in\R^{d},\end{equation}
  and \begin{equation}\label{eq3.4}\int_{B(0,r)}\int_{0}^{\infty}f(t)\exp\left[-\frac{t}{16}\inf_{x\in\R^{d}}{\rm Re}\,q(x,\xi)\right]dt\,d\xi<\infty\quad\textrm{for some}\ r>0.\end{equation}
\end{theorem}
\begin{proof}We follow the proof of \cite[Theorem 1.1]{Schilling-Wang-2013} and prove that, under the above assumptions,
$\mathbb{E}^{x}[L_{O_x}]<\infty$
for all $x\in\R^{d}$ and all open bounded neighborhoods $O_x\subseteq\R^{d}$ of $x$.
According to Lemma \ref{lm3.3}, it suffices to prove that $$\int_0^{\infty}f(t)\mathbb{P}^{x}(F_t\in O_x)dt<\infty$$ for all $x\in\R^{d}$ and all open bounded neighborhoods $O_x\subseteq\R^{d}$ of $x$. We proceed, and use the same notation, as in the proof of Theorem \ref{tm3.4}.
Fix $x=(x_1,\ldots, x_d)\in\R^{d}$ and $a>0$. Define
$$g_j(u):=e^{ix_ju}\left(1-\frac{|u|}{a}\right)1_{\left(-a,a\right)}(u),\quad u\in\R,\ j=1,\ldots,d,$$ and
$$h(y):=g_1(y_1)\cdots g(y_d),\quad y=(y_1,\ldots,y_d)\in\R^{d}.$$
Next, since $$\mathcal{F}^{-1}\left(\frac{\sin^{2}\left(\frac{a(\cdot-x_j)}{2}\right)}{a\pi^{2}
(\cdot-x_j)^{2}}\right)(u)=g_j(u),\quad u\in\R,\ j=1,\ldots,d,$$
 we conclude  $$\mathcal{F}(h)(\xi)=\frac{\sin^{2}\left(\frac{a(\xi_1-x_1)}{2}\right)}{a\pi^{2}
(\xi_1-x_1)^{2}}\cdots\frac{\sin^{2}\left(\frac{a(\xi_d-x_d)}{2}\right)}{a\pi^{2}
(\xi_d-x_d)^{2}},\quad \xi=(\xi_1,\ldots,\xi_d)\in\R^{d}.$$ Here, $\mathcal{F}^{-1}$ denotes the inverse Fourier transform.
Now, since $\sin u/u\geq1/2$ for all $|u|\leq\pi/3,$ we have $$\mathcal{F}(h)(\xi)\geq\frac{a^{d}}{4^{d}\pi^{2d}}, \quad \xi\in\left(x-2\pi/3a,x+2\pi/3 a\right)^{d},$$ where $\left(x-2\pi/3a,x+2\pi/3 a\right)^{d}:=\left(x_1-2\pi/3a,x_1+2\pi/3a\right)\times\cdots\times\left(x_d-2\pi/3a,x_d+2\pi/3a\right).$
 According to this and the monotone convergence theorem,
$$\frac{a^{d}}{8^{d}\pi^{2d}}\int_0^{\infty}f(t)\mathbb{P}^{x}\left(F_t\in \left(x-2\pi/3a,x+2\pi/3a\right)^{d}\right)dt\leq\lim_{\alpha\longrightarrow0}\int_0^{\infty}e^{-\alpha
t}f(t)\mathbb{E}^{x}\left[\mathcal{F}(h)(F_t)\right]dt.$$
Further, \cite[Theorem 1.1]{Jacob-1998} implies that $$\mathbb{E}^{x}\left[\mathcal{F}(h)(F_t)\right]=\frac{1}{(2\pi)^{d}}\int_{\R^{d}}e^{-i\langle x,y\rangle}h(y)\Phi_{t}(x,y)dy,\quad t\geq0,$$ which yields
\begin{align*}&\frac{a^{d}}{(4\pi)^{d}}\int_0^{\infty}f(t)\mathbb{P}^{x}\left(F_t\in \left(x-2\pi/3a,x+2\pi/3a\right)^{d}\right)dt\\&\leq\lim_{\alpha\longrightarrow0}\int_0^{\infty}\int_{\R^{d}}e^{-\alpha
t}f(t)e^{-i\langle x,y\rangle}h(y)\Phi_{t}(x,y)dy\,dt\\&=
\lim_{\alpha\longrightarrow0}\int_0^{\infty}\int_{\R^{d}}e^{-\alpha
t}f(t)e^{-i\langle x,y\rangle}h(y){\rm Re}\,\Phi_{t}(x,y)dy\,dt\\&\leq
\lim_{\alpha\longrightarrow0}\int_0^{\infty}\int_{\R^{d}}e^{-\alpha
t}f(t)e^{-i\langle x,y\rangle}h(y)|\Phi_{t}(x,y)|dy\,dt\\&\leq\lim_{\alpha\longrightarrow0}\int_0^{\infty}\int_{(-a,a)^{d}}e^{-\alpha
t}f(t)|\Phi_{t}(x,y)|dy\,dt\\&\leq\lim_{\alpha\longrightarrow0}\int_0^{\infty}\int_{B(0,a\sqrt{d})}e^{-\alpha
t}f(t)|\Phi_{t}(x,y)|dy\,dt
.\end{align*} Recall that $\Phi_t(x,y)=\mathbb{E}^{x}[\it{e}^{\it{i}\langle y, F_t-x\rangle}]$ for $t\geq0$ and $x,y\in\R^{d}.$
Now, due to \cite[Theorem 2.7]{Schilling-Wang-2013}, under \eqref{eq0} we have that $$|\Phi_{t}(x,y)|\leq\exp\left[-\frac{t}{16}\inf_{z\in\R^{d}}{\rm Re}\,q(z,2y)\right],\quad t\geq0,\ y\in\R^{d}.$$
Thus, \begin{align*}&\int_0^{\infty}f(t)\mathbb{P}^{x}\left(F_t\in \left(x-2\pi/3a,x+2\pi/3a\right)^{d}\right)dt\\&\leq\frac{(4\pi)^{d}}{a^{d}}\int_{B(0,a\sqrt{d})}\int_0^{\infty}f(t)\exp\left[-\frac{t}{16}\inf_{x\in\R^{d}}{\rm Re}\,q(x,2y)\right]dt\,dy\\&=\frac{(2\pi)^{d}}{a^{d}}\int_{B(0,2a\sqrt{d})}\int_0^{\infty}f(t)\exp\left[-\frac{t}{16}\inf_{x\in\R^{d}}{\rm Re}\,q(x,y)\right]dt\,dy
,\end{align*}
which completes the proof.
\end{proof}

Let us remark  that when $f(t)$ is a constant function, then the conditions in \eqref{eq3.1} and \eqref{eq3.4} become Chung-Fuchs type conditions for recurrence and transience, respectively, derived in \cite[Theorem 1.2]{Sandric-TAMS-2014} and \cite[Theorem 1.1]{Schilling-Wang-2013}.
Further, observe that if the condition in \eqref{eq3.1} holds for some $r_0>0$, then it also holds for all $r\geq r_0$. On the other hand, if \eqref{eq3.4} holds for some $r_0>0$, than it also holds for all $r\leq r_0$.
Clearly,  if we assume that
$$\inf_{r\leq|\xi|\leq
r_0}\sup_{x\in\R^{d}}|q(x,\xi)|>0,\quad 0<r< r_0,$$
then \eqref{eq3.1} does not depend on $r>0$. Recall, a function $f:\R^{d}\longrightarrow\R$ is said to be radial if for any orthogonal $d\times d$ matrix $O$, $f(x)=f(Ox)$ for all $x\in\R^{d}$.
\begin{proposition}\label{p3.8}Let $\process{F}$ be a L\'evy-type  process with
 symbol $q(x,\xi)$.
 If the function
$\xi\longmapsto\sup_{x\in\R^{d}}|q(x,\xi)|$ is radial, then the condition in \eqref{eq3.1} does not depend on $r>0$.
\end{proposition}
\begin{proof}
Let $r_0>0$ be such that
$$\int_{B(0,r_0)}\int_{0}^{\frac{\ln 2}{4\sup_{x\in\R^{d}}|q(x,\xi)|}}f(t)dt\,d\xi<\infty$$
and
$$\int_{B(0,r)}\int_{0}^{\frac{\ln 2}{4\sup_{x\in\R^{d}}|q(x,\xi)|}}f(t)dt\,d\xi=\infty,\quad r>r_0.$$ In particular,
$$\int_{B(0,r)\setminus B(0,r_0)}\int_{0}^{\frac{\ln 2}{4\sup_{x\in\R^{d}}|q(x,\xi)|}}f(t)dt\,d\xi=\infty,\quad r>r_0.$$ Now, according to (the proof of) \cite[Proposition 2.4]{Sandric-TAMS-2014}, the above relation (together with the radiality property of the function $\xi\longmapsto\sup_{x\in\R^{d}}|q(x,\xi)|$) implies that for any $\xi_0\in\R^{d}$, $|\xi_0|=r_0$, the function $\xi\longmapsto\sup_{x\in\R^{d}}|q(x,\xi)|$ is continuous and periodic with period $\xi_0$. Consequently, $q(x,\xi)=0$ for all $x,\xi\in\R^{d}$, which is impossible.
\end{proof}

Finally, recall that if $\process{L}$ is a symmetric
L\'evy process (that is, $\process{L}\stackrel{\hbox{\scriptsize{$\textrm{d}$}}}{=} \process{-L}$),  then $q(\xi)={\rm Re}\, q(\xi)$ and
$\mathbb{E}^{x}[e^{i\langle\xi, L_t-x\rangle}]={\rm Re}\,\mathbb{E}^{x}[e^{i\langle\xi, L_t-x\rangle}]=e^{-tq(\xi)}>0$ for all $t\geq0$ and all $x,\xi\in\R^{d}.$
Hence, if, in addition, $\process{L}$ is open-set irreducible,  Theorems \ref{tm3.4} and \ref{tm3.5} apply.
However, as we commented in the first section, because of the stationarity and independence
of the increments, the assumptions on the open-set irreducibility is not needed in order to introduce the notion of the $f$-weak and $f$-strong transience, and therefore to derive Chung-Fuchs type conditions for these properties, of L\'evy processes.
 By completely the same reasoning as in the proofs of Theorems \ref{tm3.4} and \ref{tm3.5} we easily  obtain the following criterion for the
 $f$-weak and $f$-strong transience of symmetric L\'evy processes (see also \cite[Theorem 3.4]{Sato-Watanabe-2004}).
 \begin{corollary}\label{c3.9}Let  $f:[0,\infty)\longrightarrow[0,\infty)$ be a  non-decreasing and continuously differentiable function.
Then, a symmetric and transient L\'evy process $\process{L}$ with symbol $q(\xi)$ is $f$-weakly transient if, and only if,
$$\int_{B(0,r)}\int_{0}^{\infty}f(t)e^{-tq(\xi)}dt\,d\xi=\infty\quad\textrm{for some (all)}\ r>0.$$
\end{corollary}

\section{Algebraic Weak and Strong Transience of L\'evy-Type Processes}\label{s3}
In this section, we  concentrate on the $f$-weak and $f$-strong transience with respect to $f(t):=t^{\kappa}$ for  some $\kappa>0$ and we use the terminology $\kappa$-weak and $\kappa$-strong transience, respectively.
Obviously, the Chung-Fuchs type conditions derived in Theorems \ref{tm3.4} and \ref{tm3.5} now read
\begin{align}\label{eq4.1}&\int_{B(0,r)}\frac{d\xi}{\left(\sup_{x\in\R^{d}}|q(x,\xi)|\right)^{\kappa+1}}=\infty\quad\textrm{for some}\ r>0\\
\label{eq4.3}&\int_{B(0,r)}\frac{d\xi}{\left(\inf_{x\in\R^{d}}{\rm Re}\,q(x,\xi)\right)^{\kappa+1}}<\infty\quad\textrm{for some}\ r>0,\end{align}
respectively.
As a simple consequence, we also get the following Chung-Fuchs type condition for the $\kappa$-strong transience.
\begin{corollary}\label{c4.1}Let $\kappa>0$ and let $\process{F}$ be a L\'evy-type process   with  symbol $q(x,\xi)$ satisfying $|{\rm Im}\, q(x,\xi)|\leq c\, {\rm Re}\, q(x,\xi)$ for some $c\geq0$ and all $x,\xi\in\R^{d}$. Then, the condition in \eqref{eq4.1} is equivalent to
$$\int_{B(0,r)}\frac{d\xi}{\left(\sup_{x\in\R^{d}}{\rm Re}\,q(x,\xi)\right)^{\kappa+1}}<\infty\quad\textrm{for some}\ r>0,$$
and the condition in \eqref{eq4.3} is equivalent to
$$\int_{B(0,r)}\frac{d\xi}{\left(\inf_{x\in\R^{d}}|q(x,\xi)|\right)^{\kappa+1}}<\infty\quad\textrm{for some}\ r>0.$$
\end{corollary}
\begin{proof}
The assertion easily follows from the following inequalities
$$\rm{Re}\,\it{q}(x,\xi)\leq\sqrt{(\rm{Re}\,\it{q}(x,\xi))^{\rm{2}}+(\rm{Im}\,\it{q}(x,\xi))^{\rm{2}}}=|q(x,\xi)|\leq \sqrt{{\rm 1+{\it c}^{2}}}\,\rm{Re}\,\it{q}(x,\xi),\quad x,\xi\in\R^{d}.$$
\end{proof}

In Proposition \ref{p3.8} we discussed conditions under which the relation in \eqref{eq3.1} (and hence in \eqref{eq4.1}) does not depend on $r>0$. By the same reasoning as before, we get that \eqref{eq4.3} does not depend on $r>0$ if
 $$\inf_{r_0\leq|\xi|\leq
r}\inf_{x\in\R^{d}}{\rm Re}\,q(x,\xi)>0,\quad r> r_0>0.$$
Let us remark that the above condition is
 satisfied if
$$\inf_{|\xi|=1}\inf_{x\in\R^{d}}\left(\langle\xi,C(x)\xi\rangle+\int_{B\left(0,\frac{1}{r}\right)}\langle\xi,y\rangle^{2}\nu(x,dy)\right)>0,\quad r>r_0>0,$$ (see \cite{Sandric-TAMS-2014}).
In addition, if we assume
that
 the function
$\xi\longmapsto\inf_{x\in\R^{d}}\sqrt{\rm{Re}\,\it{q}(x,\xi)}$ is
radial and subadditive (that is, $\inf_{x\in\R^{d}}\sqrt{\rm{Re}\,\it{q}(x,\xi+\eta)}\leq\inf_{x\in\R^{d}}\sqrt{\rm{Re}\,\it{q}(x,\xi)}+\inf_{x\in\R^{d}}\sqrt{\rm{Re}\,\it{q}(x,\eta)}$ for all $\xi,\eta\in\R^{d}$), then the condition in
\eqref{eq4.3} do not depend on $r>0$ (see \cite[Proposition 2.4]{Sandric-TAMS-2014}).

 In
the following proposition we discuss the dependence of the $\kappa$-weak and $\kappa$-strong transience on  the dimension of the state space.
\begin{proposition}\label{p4.2} Let $\kappa>0$ and let $\process{F}$ be a L\'evy-type  process with  symbol $q(x,\xi)$.
\begin{enumerate}
  \item [(i)]If
$$\limsup_{|\xi|\longrightarrow0}\frac{\sup_{x\in\R^{d}}|q(x,\xi)|}{|\xi|^{\gamma}}<\infty$$
and $d\leq(\kappa+1)\gamma$ for
 some $\gamma>0$, then \eqref{eq4.1} holds true.
 \item [(ii)]If
$$\liminf_{|\xi|\longrightarrow0}\frac{\inf_{x\in\R^{d}}\rm{Re}\,\it{q(x,\xi)}}{|\xi|^{\gamma}}>0$$
and $d>(\kappa+1)\gamma$ for
 some $\gamma>0$,
   then \eqref{eq4.3} holds true.
\end{enumerate}
\end{proposition}
\begin{proof}\begin{enumerate}
               \item [(i)]By assumption, there exist constants $c>0$ (large enough) and  $r>0$ (small enough), such that
$$\int_{B(0,r)}\frac{d\xi}{\left(\sup_{x\in\R^{d}}|q(x,\xi)|\right)^{\kappa+1}}\geq\frac{1}{c}\int_{B(0,r)}\frac{d\xi}{|\xi|^{(\kappa+1)\gamma}}=\frac{S_{d}}{c}\int_0^{r}\rho^{d-1-(\kappa+1)\gamma}d\rho.$$ Here, $S_d$ denotes the surface of a $d$-dimensional unit ball.
            \item [(ii)]By assumption, there exist constants $c>0$  and  $r>0$ (small enough), such that
$$\int_{B(0,r)}\frac{d\xi}{(\inf_{x\in\R^{d}}{\rm Re}\,q(x,\xi))^{\kappa+1}}\leq\frac{1}{c}\int_{B(0,r)}\frac{d\xi}{|\xi|^{(\kappa+1)\gamma}}=\frac{S_{d}}{c}\int_0^{r}\rho^{d-1-(\kappa+1)\gamma}d\rho.$$
             \end{enumerate}
\end{proof}
As a direct consequence of the above proposition we get the following.
\begin{theorem} \label{tm4.3}Let $\kappa>0$ and let $\process{F}$ be a L\'evy-type  process  with  symbol
$q(x,\xi)$ and L\'evy triplet $(b(x),C(x),\nu(x,dy))$. \begin{itemize}
          \item [(i)]
If  $q(x,\xi)=q(x,-\xi)$ for all $x,\xi\in\R^{d}$ (that is, $b(x)=0$ and $\nu(x,dy)$ is a symmetric measure for all $x\in\R^{d}$),  $d\leq2(\kappa+1)$ and
$$\sup_{x\in\R^{d}}\int_{\R^{d}}|y|^{2}\nu(x,dy)<\infty,$$ then  \eqref{eq4.1} holds true.
  \item [(ii)] If  $d>2(\kappa+1)$ and
\begin{equation}\label{eq4.4}\liminf_{|\xi|\longrightarrow0}\frac{\inf_{x\in\R^{d}}\left(\langle\xi,C(x)\xi\rangle+\int_{\{|y|\leq
\frac{\pi}{2|\xi|}\}}\langle\xi,y\rangle^{\rm{2}}\nu(\it{x},dy)\right)}{|\xi|^{2}}>\rm{0},\end{equation}
then \eqref{eq4.3} holds true.
\end{itemize}
\end{theorem}
\begin{proof}
\begin{enumerate}
  \item [(i)]
  First, observe that \begin{align*}\sup_{x\in\R^{d}}|q(x,\xi)|&=\sup_{x\in\R^{d}}\left(\langle\xi,C(x)\xi\rangle+\int_{\R^{d}}(1-\cos\langle\xi,y\rangle)\nu(x,dy)\right)\\
  &\leq |\xi|^{2}\left(d\sup_{x\in\R^{d}}\max_{1\leq i,j\leq
d}|c_{ij}(x)|+\sup_{x\in\R^{d}}\int_{\R^{d}}|y|^{2}\nu(x,dy)\right),\end{align*}
where we used the facts that  for  an arbitrary square matrix $M=(m_{ij})_{1\leq
i,j\leq d}$ and $v\in\R^{d}$,  $|\langle
v,Mv\rangle|\leq|v||Mv|\leq d\max_{1\leq i,j\leq
d}|m_{ij}||v|^{2}$ and $1-\cos u\leq u^{2}$ for
all $u\in\R$. Now, the claim follows as a direct consequence of Proposition \ref{p4.2}.
  \item [(ii)]   By employing the fact  that $1-\cos
u\geq u^{2}/\pi$ for all  $|u|\leq\pi/2$,
\begin{align*}\inf_{x\in\R^{d}}{\rm Re}\,q(x,\xi)&=\inf_{x\in\R^{d}}\left(\langle\xi,c(x)\xi\rangle+\int_{\R^{d}}(1-\cos\langle\xi,y\rangle)\nu(x,dy)\right)\\&\geq
\frac{1}{\pi}\inf_{x\in\R^{d}}\left(\langle\xi,c(x)\xi\rangle+\int_{\{|y|\leq
\frac{\pi}{2|\xi|}\}}\langle\xi,y\rangle^{\rm{2}}\nu(\it{x},dy)\right),\end{align*} which together with Proposition \ref{p4.2} proves the assertion.
\end{enumerate}
\end{proof}
Intuitively, the relation in \eqref{eq4.4} actually says that the underlying L\'evy-type process in non-degenerate, that is, it  either has a non-degenerate diffusion part or non-degenerate jump part.
Now, we give some applications of the above results.
\begin{example}[Elliptic diffusion processes]\label{e4.4}{\rm
Assume that the functions $b=(b_i)_{1\leq i\leq d}:\R^{d}\longrightarrow\R^{d}$ and $C=(c_{ij})_{1\leq i,j\leq d}:\R^{d}\longrightarrow\R^{d}\times\R^{d}$ satisfy the following:
 \begin{itemize}
  \item [(i)] $b(x)$ is bounded and continuous;
   \item [(ii)] $C(x)$ is   bounded, symmetric and Lipschitz continuous;
   \item [(ii)] for some constant $c>0$ and all  $\xi\in\R^{d}$, $\inf_{x\in\R^{d}}\langle \xi,C(x)\xi\rangle\geq c|\xi|^{2}.$
 \end{itemize}
  Then, according to \cite[Theorem V.24.1]{Rogers-Williams-Book-II-2000} and \cite[Theorem 2.3]{Stramer-Tweedie-1997},
  there exists a unique open-set irreducible elliptic diffusion process satisfying conditions (\textbf{C1}), (\textbf{C2}) and (\textbf{C3})   (hence a L\'evy-type  process),
  with symbol of the form  $$q(x,\xi)=- i\langle \xi,b(x)\rangle +
\frac{1}{2}\langle\xi,C(x)\xi\rangle,\quad x,\xi\in\R^{d}.$$
Fix $\kappa>0.$ Then,
\begin{itemize}
 \item[(i)] if $\sup_{x\in\R^{d}}|b(x)|>0$   and $d\leq(\kappa+1)$, Proposition \ref{p4.2} implies that \eqref{eq4.1} holds true.
  \item [(ii)]  if $b(x)=0$  for all $x\in\R^{d}$ and $d\leq2(\kappa+1)$,  Theorem \ref{tm4.3}  entails that \eqref{eq4.1} holds true.
  \item [(iii)]  if  $d>2(\kappa+1)$, Theorem \ref{tm4.3} implies that \eqref{eq4.3} holds true.
\end{itemize}
  Let us remark  that if $d\leq2$ and $b(x)=0$  for all $x\in\R^{d}$, then the underlying elliptic diffusion process satisfies the Chung-Fuch type condition for recurrence given in \cite[Theorem 1.2]{Sandric-TAMS-2014} and if $d\geq3$ it is always transient (see \cite[Theorem 2.9]{Sandric-TAMS-2014}).
In particular,  a standard $d$-dimensional, $d\geq3$,  zero drift Brownian motion is $\kappa$-weakly transient if, and only if, $d\leq2(\kappa+1)$.
}
\end{example}

\begin{example}[Stable-like processes]\label{e4.5}
\rm{
Let $\alpha:\R^{d}\longrightarrow(0,2)$, $\beta:\R^{d}\longrightarrow\R^{d}$ and $\gamma:\R^{d}\longrightarrow(0,\infty)$  be such that:
\begin{itemize}
  \item [(i)] $\alpha,\beta,\gamma\in C^{1}_b(\R^{d})$;
  \item [(ii)]$0<\underline{\alpha}:=\inf_{x\in\R^{d}}\alpha(x)\leq\sup_{x\in\R^{d}}\alpha(x)=:\overline{\alpha}<2$ and $\inf_{x\in\R^{d}}\gamma(x)>0$.
\end{itemize}
Here, $C_b^{k}(\R^{d})$, $k\geq 0$, denotes the space of $k$ times differentiable functions such that all derivatives up to order $k$ are bounded.
  Then, under this assumptions, in
\cite{Bass-1988}, \cite[Theorem 5.1]{Kolokoltsov-2000} and
                                                         \cite[Theorems 1.1 and 3.3]{Schilling-Wang-2013}
                                                         it has been shown
                                                         that there
                                                         exists a
                                                         unique open-set irreducible L\'evy-type
                                                         process $\process{F}$,
                                                         called a
                                                         \emph{stable-like
                                                         process},
                                                         determined
by a symbol of the form $q(x,\xi)=-i\langle\xi,\beta(x)\rangle+\gamma(x)|\xi|^{\alpha(x)}$ for  $x,\xi\in\R^{d}$.  Note that when
$\alpha(x)$, $\beta(x)$ and $\gamma(x)$ are constant functions, then we deal
with a rotationally invariant  stable L\'evy process with drift.
Again, fix $\kappa>0$. Then,   Proposition \ref{p4.2} implies that
\begin{itemize}
  \item [(i)] if $\sup_{x\in\R^{d}}|\beta(x)|>0$, $\underline{\alpha}<1$ and $d\leq(\kappa+1)\underline{\alpha}$, \eqref{eq4.1} holds true.
  \item [(ii)]if $\sup_{x\in\R^{d}}|\beta(x)|>0$, $\underline{\alpha}\geq1$ and $d\leq(\kappa+1)$, \eqref{eq4.1} holds true.
  \item [(iii)]if $\beta(x)=0$ for all $x\in\R^{d}$  and $d\leq(\kappa+1)\underline{\alpha}$, \eqref{eq4.1} holds true.
  \item [(iv)]if $d>(\kappa+1)\overline{\alpha}$, \eqref{eq4.3} holds true.
\end{itemize}
Let also remark that
 \begin{enumerate}
   \item [(i)] if $d=1$, $\beta(x)=0$ for all $x\in\R^{d}$  and $\underline{\alpha}\geq1$, then $\process{F}$ satisfies the Chung-Fuchs type condition for recurrence given in \cite[Theorem 1.2]{Sandric-TAMS-2014}.
   \item [(ii)] if $d=1$ and $\overline{\alpha}<1$, then $\process{F}$ is transient.
   \item [(iii)] if $d\geq2$, then $\process{F}$ is always transient.
 \end{enumerate}
 (see \cite[Theorem 2.10]{Sandric-TAMS-2014}).
In particular,  a  $d$-dimensional rotationally invariant  $\alpha$-stable L\'evy process is $\kappa$-weakly transient if, and only if, $d/(\kappa+1)\leq\alpha<d$.
}
 \end{example}

 The concept of the indices of stability  can
be generalized to  general L\'evy-type process through the so-called
Pruitt indices (see \cite{Pruitt-1981}). The \emph{Pruitt indices}, for a
L\'evy-type process $\process{F}$ with  symbol $q(x,\xi)$, are
defined in the following way
\begin{align*}\underline{\delta}&:=\sup\left\{\delta\geq0:\lim_{|\xi|\longrightarrow0}\frac{\sup_{x\in\R^{d}}|q(x,\xi)|}{|\xi|^{\delta}}=0\right\}\\
\overline{\delta}&:=\sup\left\{\delta\geq0:\lim_{|\xi|\longrightarrow0}\frac{\inf_{x\in\R^{d}}\rm{Re}\,\it{q}(\it{x},\xi)}{|\xi|^{\delta}}=\rm{0}\right\}\end{align*}
(see \cite{Schilling-PTRF-1998}). Note that $0\leq\underline{\delta}\leq\overline{\delta}$,
$\underline{\delta}\leq2$ and in the case of a  stable-like processes, we have
$\underline{\delta}=\underline{\alpha}$ and $\overline{\delta}=\bar{\alpha}.$  Now, we
 generalize Example
  \ref{e4.5} in terms of the Pruitt indices.
\begin{theorem}\label{tm4.6}Let $\kappa>0$ and let $\process{F}$ be a   L\'evy-type process with  symbol $q(x,\xi)$.
\begin{enumerate}
  \item [(i)] If   $d<(\kappa+1)\underline{\delta}$, then  the condition in \eqref{eq4.1} holds true.
\item [(ii)] If  $q(x,\xi)$
satisfies the condition in  \eqref{eq4.3}, then $d\geq(\kappa+1)\bar{\delta}$.
\end{enumerate}
\end{theorem}
\begin{proof} \begin{enumerate}
                \item [(i)] Let $d/(\kappa+1)\leq\delta<\underline{\delta}$ be arbitrary. Then, by the definition of
$\underline{\delta}$,
$$\lim_{|\xi|\longrightarrow0}\frac{\sup_{x\in\R^{d}}|q(x,\xi)|}{|\xi|^{\delta}}=0.$$
Now, the claim easily follows by employing completely the same arguments as in the proof of  Proposition \ref{p4.2}.
\item[(ii)]
   Let us
assume that this is not the case.  Then, for all
$d/(\kappa+1)\leq\delta<\overline{\delta}$, by the definition of $\overline{\delta}$, we have
$$\lim_{|\xi|\longrightarrow0}\frac{\inf_{x\in\R^{d}}\rm{Re}\,\it{q}(x,\xi)}{|\xi|^{\delta}}=\rm{0}.$$
Consequently,
$$\int_{B(0,r)}\frac{d\xi}{\left(\inf_{x\in\R^{d}}\rm{Re}\,\it{q}(x,\xi)\right)^{\kappa+1}}\geq\int_{B(0,r)}\frac{d\xi}{|\xi|^{(\kappa+1)\delta}}=\infty$$
for all $r>0$ small enough.
\end{enumerate}
\end{proof}
Let us remark that in general it is not possible to obtain the equivalence in Theorem \ref{tm4.6}. To see this, first recall that for two
symmetric measures $\mu(dx)$ and $\bar{\mu}(dx)$ on
$\mathcal{B}(\R)$ which are finite outside of any neighborhood around
the origin, we say that $\mu(dx)$ has a bigger tail than
$\bar{\mu}(dx)$ if there exists $x_0>0$ such that
$\mu(x,\infty)\geq\bar{\mu}(x,\infty)$ for all $x\geq x_0$. Now, fix $\kappa>0$ and
let $\bar{\nu}(dx)$ be the L\'evy
measure of a one-dimensional symmetric   $\alpha$-stable, $\alpha(\kappa+1)<1$,  L\'evy process
$\process{\bar{L}}$. Hence, $\process{\bar{L}}$ is $\kappa$-strongly transient.  Next,
by
\cite[Theorem 38.4]{Sato-Book-1999},  there exists a  one-dimensional
symmetric L\'evy process $\process{L}$ with  L\'evy measure
$\nu(dx)$ having a bigger tail then $\bar{\nu}(dx)$ and satisfying the condition in \eqref{eq4.1}. Moreover, $\process{L}$  can be constructed such that it  is recurrent. Consequently, by
Fubini's theorem, for any $\delta>0$ we have
\begin{align*}\int_{\{x>x_0\}}x^{\delta}\nu(dx)&=\int_{\{x>x_0\}}\int_0^{x}\delta
y^{\delta-1}dy\,\nu(dx)\\&=x_0^{\delta}\nu(x_0,\infty)+\delta\int_{\{y>x_0\}}y^{\delta-1}\nu(y,\infty)dy\\
&\geq
x_0^{\delta}\nu(x_0,\infty)+\delta\int_{\{y>x_0\}}y^{\delta-1}\bar{\nu}(y,\infty)dy\\
&\geq x_0^{\delta}\nu(x_0,\infty)-x_0^{\delta}\bar{\nu}(x_0,\infty)+\int_{\{x>x_0\}}x^{\delta}\bar{\nu}(dx).\end{align*}
Hence, if $\int_{\{x>1\}}x^{\delta}\bar{\nu}(dy)=\infty$, then
$\int_{\{x>1\}}x^{\delta}\nu(dy)=\infty.$
Specially, according to \cite[Prposition 5.4]{Schilling-PTRF-1998}, $\underline{\delta}_{L}\leq\underline{\delta}_{\bar{L}}=\alpha$.
To see that $d\geq\overline{\delta}(\kappa+1)$ does not automatically imply \eqref{eq4.3}, we proceed in a similar way.
 Fix $\kappa>0$ and
let $\bar{\nu}(dx)$ be the L\'evy
measure of a  one-dimensional symmetric $\alpha$-stable, $\alpha(\kappa+1)<1$,  L\'evy process
$\process{\bar{L}}$.   Again,
by
\cite[Theorem 38.4]{Sato-Book-1999},  there exists a  one-dimensional symmetric
 L\'evy process $\process{L}$ with  L\'evy measure
$\nu(dx)$ having a bigger tail then $\bar{\nu}(dx)$ and satisfying the condition in \eqref{eq4.1}. But, by the same reasoning as above,
we conclude that $\overline{\delta}_{L}\leq\overline{\delta}_{\bar{L}}=\alpha$.
However, by assuming certain regularities (convexity and
concavity) of the  symbol, we get the converse statements in Theorem
\ref{tm4.6}.

 \begin{theorem}\label{tm4.7}Let $\kappa>0$ and let $\process{F}$ be a  L\'evy-type  process with  symbol
 $q(x,\xi)$.
\begin{enumerate}
\item [(i)]If  $\kappa+1\geq d$ and the functions $\xi\longmapsto\sup_{x\in\R^{d}}|\it{q}(\it{x},\xi)|$ and $|\xi|\longmapsto\sup_{x\in\R^{d}}|\it{q}(\it{x},\xi)|$ are radial and convex, respectively, on some neighborhood of the origin,  then $\underline{\delta}(\kappa+1)\geq d$.
\item[(ii)]If  $\kappa+1\leq d$ and  the functions $\xi\longmapsto\sup_{x\in\R^{d}}|\it{q}(\it{x},\xi)|$ and $|\xi|\longmapsto\sup_{x\in\R^{d}}|\it{q}(\it{x},\xi)|$ are radial and concave, respectively, on some neighborhood of the origin,  then $\underline{\delta}(\kappa+1)\leq d$. In addition, if  the condition in \eqref{eq4.1} holds true, then $d=\kappa+1.$
\item [(iii)]If $\kappa+1\geq d$ and the functions
$\xi\longmapsto\inf_{x\in\R^{d}}\rm{Re}\,\it{p}(\it{x},\xi)$ and $|\xi|\longmapsto\inf_{x\in\R^{d}}\rm{Re}\,\it{p}(\it{x},\xi)$  are
radial and convex, respectively, on some neighborhood of the origin, then $\overline{\delta}(\kappa+1)\geq d$.  In addition, if  the condition in \eqref{eq4.3} holds true, then $d=\kappa+1.$
\item[(iv)]If $\kappa+1\leq d$ and the functions $\xi\longmapsto\inf_{x\in\R^{d}}\rm{Re}\,\it{p}(\it{x},\xi)$ and $|\xi|\longmapsto\inf_{x\in\R^{d}}\rm{Re}\,\it{p}(\it{x},\xi)$ are radial
and concave, respectively,  on some neighborhood of the origin, then $\overline{\delta}(\kappa+1)\leq d$.  In addition, if $\kappa+1<d$, then the condition in \eqref{eq4.3} holds true.
\end{enumerate}
 \end{theorem}
\begin{proof}
\begin{enumerate}
\item [(i)]
 Let $\varepsilon>0$ be such that
$|\xi|\longmapsto\sup_{x\in\R^{d}}|{q}(\it{x},\xi)|$ is convex on
$[0,\varepsilon)$. Then, for  all $\delta<1$,
\begin{align*}\limsup_{|\xi|\longrightarrow0}\frac{\sup_{x\in\R^{d}}
                |\it{q(x,\xi)}|}{|\xi|^{\delta}}&=\limsup_{|\xi|\longrightarrow0}\frac{\sup_{x\in\R^{d}}
                \left|\it{q\left(x,\frac{\rm{2}|\xi|}{\varepsilon}\frac{\varepsilon\xi}{\rm{2}|\xi|}\right)}\right|}{|\xi|^{\delta}}\\&\leq \limsup_{|\xi|\longrightarrow0}\frac{2|\xi|^{1-\delta}}{\varepsilon}\sup_{x\in\R^{d}}
               \left| \it{q\left(x,\frac{\varepsilon\xi}{\rm{2}|\xi|}\right)}\right|={\rm0},\end{align*}
where in the second step we used the convexity property. Thus, $\underline{\delta}\geq1\geq d/(\kappa+1)$.
\item[(ii)]
Let $\varepsilon>0$ be such that
$|\xi|\longmapsto\sup_{x\in\R^{d}}|{q}(\it{x},\xi)|$ is concave on
$[0,\varepsilon)$. Then, for  all $\delta\geq1$,
\begin{align*}\liminf_{|\xi|\longrightarrow0}\frac{\sup_{x\in\R^{d}}
                \it{q(x,\xi)}}{|\xi|^{\delta}}&=\liminf_{|\xi|\longrightarrow0}\frac{\sup_{x\in\R^{d}}
               \left|\it{q\left(x,\frac{\rm{2}|\xi|}{\varepsilon}\frac{\varepsilon\xi}{\rm{2}|\xi|}\right)}\right|}{|\xi|^{\delta}}\\&\geq \liminf_{|\xi|\longrightarrow0}\frac{2|\xi|^{1-\delta}}{\varepsilon}\sup_{x\in\R^{d}}
                \left|\it{q\left(x,\frac{\varepsilon\xi}{\rm{2}|\xi|}\right)}\right|>{\rm0},\end{align*}
where in the second step we applied the concavity property.
Hence, $\underline{\delta}\leq1\leq d/(1+\kappa)$. The second assertion easily follows from the above relation and the same arguments employed in the proof of Proposition \ref{p4.2}.

 \item [(iii)] The proof proceeds similarly as in  (i).
\item[(iv)]  The proof proceeds similarly as in  (ii).
\end{enumerate}
\end{proof}

\section{Algebraic Weak and Strong Transience of Rotationally Invariant  L\'evy-Type Processes}\label{s4}
In this section, we discuss   the $\kappa$-weak and $\kappa$-strong transience of a  class of  L\'evy-type processes whose symbol is radial in the co-variable. In particular, if $(b(x),C(x),\nu(x,dy))$ denotes the corresponding L\'evy triplet,  \cite[Exercise 18.3]{Sato-Book-1999} implies that this property is equivalent to
\begin{itemize}
  \item [(i)]$b(x)=0$ for all $x\in\R^{d}$;
  \item [(ii)]$C(x)=c(x)I$ for some Borel measurable function $c:\R^{d}\longrightarrow[0,\infty)$, where $I$ is the $d\times d$ identity matrix;
  \item [(iii)] $\nu(x,dy)=\nu(x,O dy)$ for all $x\in\R^{d}$ and all orthogonal $d\times d$ matrices $O$.
\end{itemize}
Specially, due to Proposition \ref{p3.8}, these conditions entail that the condition in \eqref{eq4.1} does not depend on $r>0$. By following the proof of \cite[Theorem 3.2]{Sandric-SPA-2015} we get the following ``perturbation" result.

\begin{theorem}\label{tm5.1} Let $\kappa>0$ and let $\process{F}$ and $\process{\bar{F}}$ be
$d$-dimensional  L\'evy-type processes  with  symbols $q(x,\xi)$ and $\bar{q}(x,\xi)$ and  L\'evy triplets
$(0,c(x)I,\nu(x,dy))$ and $(0,\bar{c}(x)I,\bar{\nu}(x,dy))$, respectively.
If there exists
orthogonal matrix $O$ such that
\begin{equation}\label{eq5.1}\sup_{x\in\R^{d}}\int_{\R^{d}}
|y|^{2}|\nu(x,dy)-\bar{\nu}(Ox,dy)|<\infty,\end{equation}  then $q(x,\xi)$ satisfies  \eqref{eq4.1} if, and only if,  $\bar{q}(x,\xi)$ satisfies \eqref{eq4.1}. Here,
$|\mu(dy)|$
denotes the
total variation measure of the signed measure $\mu(dy)$.
If
\begin{equation}\label{eq5.2}\liminf_{|\xi|\longrightarrow0}\frac{\inf_{x\in \R^{d}}q(x,\xi)}{|\xi|^{2}}
>\frac{1}{2}\sup_{x\in\R^{d}}|c(x)-\bar{ c}(O x)|+\sup_{x\in\R^{d}}\int_{\R^{d}}
|y|^{2}|\nu(x,dy)-\bar{\nu}(O x,dy)|,\end{equation}then, under \eqref{eq5.1},
$q(x,\xi)$ satisfies  \eqref{eq4.3} if, and only if,  $q(x,\xi)$ satisfies \eqref{eq4.3}.
\end{theorem}

Let us remark here that if $\nu(x,dy)$ is the L\'evy measure of a L\'evy-type process $\process{F}$, then $\nu(Ox,dy)=\nu(Ox,Ody)$ is a L\'evy measure of the L\'evy-type process $\process{O^{-1}F}$ (see \cite[Proposition 3.1]{Sandric-SPA-2015}). Further, observe that in the L\'evy process case the condition in \eqref{eq5.2} will be satisfied if, and only if, $\int_{\R^{d}}|y|^{2}\nu(dy)=\infty.$
A situation where the condition in \eqref{eq5.1} trivially holds true is given in
the following proposition.
\begin{proposition}\label{p5.2}
 Let $\process{F}$ and $\process{\bar{F}}$ be
$d$-dimensional Le\'vy-type processes with  L\'evy
measures $\nu(x,dy)$ and $\bar{\nu}(x,dy)$, respectively. If there
exist an orthogonal matrix $O$ and $r>0$,  such that $\nu(x,B)=\bar{\nu}(Ox,B)$  for all $x\in\R^{d}$ and all $B\in\mathcal{B}(\R^{d})$, $B\subseteq B^{c}(0,r)$, then  the condition in
\eqref{eq5.1} holds true.
\end{proposition}

Note that Proposition \ref{p5.2} implies that the $\kappa$-weak and $\kappa$-strong transience of L\'evy-type processes, satisfying the conditions from Theorem \ref{tm5.1}, depend only on big jumps, that is, they do not depend on the continuous part of the process and small jumps.

In the sequel, we derive some conditions for the $\kappa$-weak and $\kappa$-strong transience in terms of the L\'evy measures.
 By a straightforward adaptation of \cite[Theorem 4.4]{Sandric-SPA-2015}, we conclude the following.

\begin{theorem}\label{tm5.3} Let $\kappa>0$ and let $\process{F}$  be a
 L\'evy-type process with   L\'evy measure $\nu(x,dy)$.
 Then,
\begin{equation}\label{eq5.5}\int_r^{\infty}\frac{\rho^{2\kappa-d+1}}{\left(\sup_{x\in\R^{d}}\int_0^{\rho}u\, \nu(x,B^{c}(0,u))du\right)^{\kappa+1}}d\rho=\infty\quad\textrm{for some (all)}\ r>0\end{equation}
implies \eqref{eq4.1}, and \eqref{eq4.3} implies
\begin{equation}\label{eq5.6}\int_r^{\infty}\frac{\rho^{2\kappa-d+1}}{\left(\inf_{x\in\R^{d}}\int_0^{\rho}u\,\nu(x,B^{c}(0,u))du\right)^{\kappa+1}}d\rho<\infty\quad \textrm{for some}\ r>0.\end{equation}
In addition,  if the condition in \eqref{eq5.2} holds and there exists $u_0\geq0$ such that $\nu(x,dy)=n(x,|y|)dy$ on $\mathcal{B}(B^{c}(0,u_0))$ for some Borel function $n:\R^{d}\times(0,\infty)\longrightarrow(0,\infty)$ which is decreasing on $(u_0,\infty)$ for all $x\in\R^{d}$,
 then
\eqref{eq4.1} holds true if, and only if, \eqref{eq5.5} holds true and \eqref{eq4.3} holds true if, and only if, \eqref{eq5.6} holds true.
\end{theorem}
Let us remark that in general  we cannot conclude the equivalence in Theorem \ref{tm5.3} (see the discussion after Theorem \ref{tm4.6}).
 By employing the integration by parts formula and Fatou's lemma, for any $x\in\R^{d}$, any $\rho>0$ and any $0<\varepsilon<\rho$, we have
\begin{align*}&\frac{\rho^{2}}{2}\nu(x,B^{c}(0,\rho))+\frac{1}{2}\int_{B(0,\rho)}|y|^{2}\,\nu(x,dy)\\
&\geq\liminf_{\epsilon\longrightarrow0}\left(\frac{\rho^{2}}{2}\nu(x,B^{c}(0,\rho))-\frac{\epsilon^{2}}{2}\nu(x,B^{c}(0,\epsilon))+\frac{1}{2}\int_{B(0,\rho)\cap´B^{c}(0,\epsilon)}|y|^{2}\,\nu(x,dy)\right)\\
&=\liminf_{\epsilon\longrightarrow0}\int_\epsilon^{\rho}u\,\nu(x,B^{c}(0,u))du\\
&\geq\int_0^{\rho}u\,\nu(x,B^{c}(0,u))du\\&=\int_0^{\varepsilon}u\,\nu(x,B^{c}(0,u))du+\int_{\varepsilon}^{\rho}u\,\nu(x,B^{c}(0,u))du\\
&\geq\frac{\varepsilon^{2}}{2}\nu(x,B^{c}(0,\varepsilon))+\frac{\rho^{2}}{2}\nu(x,B^{c}(0,\rho))-\frac{\varepsilon^{2}}{2}\nu(x,B^{c}(0,\varepsilon))+\frac{1}{2}\int_{B(0,\rho)\cap´B^{c}(0,\varepsilon)}|y|^{2}\,\nu(x,dy)\\
&=\frac{\rho^{2}}{2}\nu(x,B^{c}(0,\rho))+\frac{1}{2}\int_{B(0,\rho)\cap´B^{c}(0,\varepsilon)}|y|^{2}\,\nu(x,dy).\end{align*} Now, by letting $\varepsilon\longrightarrow0$, Fatou's lemma yields $$\int_0^{\rho}u\,\nu(x,B^{c}_u(0))du=\frac{\rho^{2}}{2}\nu(x,B^{c}_\rho(0))+\frac{1}{2}\int_{B_\rho(0)}|y|^{2}\,\nu(x,dy).$$
Thus, we have proved the following.

\begin{proposition}\label{p5.4}Let $\kappa>0$ and let $\process{F}$  be a
 L\'evy-type process with   L\'evy measure $\nu(x,dy)$.
Then,
\eqref{eq5.5} holds  if
\begin{equation}\label{eq5.10}\int_{r}^{\infty}\frac{\rho^{2\kappa-d+1}}{\left(\rho^{2}\sup_{x\in\R^{d}}\nu(x,B^{c}(0,\rho))+\sup_{x\in\R^{d}}\int_{B(0,\rho)}|y|^{2}\,\nu(x,dy)\right)^{\kappa+1}}d\rho=\infty\quad\textrm{for some}\ r>0,\end{equation} and \eqref{eq5.6} holds  if \begin{equation}\label{eq5.11}\int_{r}^{\infty}\frac{\rho^{2\kappa-d+1}}{\left(\rho^{2}\inf_{x\in\R^{d}}\nu(x,B^{c}(0,\rho))+\inf_{x\in\R^{d}}\int_{B(0,\rho)}|y|^{2}\,\nu(x,dy)\right)^{\kappa+1}}d\rho<\infty\quad\textrm{for some}\ r>0.\end{equation} In particular, \eqref{eq5.6} holds if either one of the following  conditions holds
\begin{equation}\label{eq5.12}\int_{r}^{\infty}\frac{\rho^{-d-1}}{\left(\sup_{x\in\R^{d}}\nu(x,B^{c}(0,\rho))\right)^{\kappa+1}}d\rho<\infty\quad\textrm{for some}\ r>0\end{equation} or
\begin{equation}\label{eq5.13}\int_{r}^{\infty}\frac{\rho^{2\kappa-d+1}}{\left(\inf_{x\in\R^{d}}\int_{B(0,\rho)}|y|^{2}\,\nu(x,dy)\right)^{\kappa+1}}d\rho<\infty\quad\textrm{for some}\ r>0.\end{equation}
\end{proposition}
As a direct consequence of the above proposition we get the following.
\begin{corollary}\label{c5.5}Let $\kappa>0$ and  let $\process{F}$  be a
 L\'evy-type  process with    L\'evy measure $\nu(x,dy)$. Assume that there exist $u_0\geq0$ such that $\nu(x,dy)=n(x,|y|)dy$ on $\mathcal{B}(B^{c}(0,u_0))$ for some Borel function $n:\R^{d}\times(0,\infty)\longrightarrow(0,\infty)$ which is decreasing on $(u_0,\infty)$ for all $x\in\R^{d}.$ Then,
\begin{equation}\label{eq5.14}\int_{r}^{\infty}\frac{\rho^{-d\kappa-2d-1}d\rho}{\left(\inf_{x\in\R^{d}}n(x,\rho)\right)^{\kappa+1}}<\infty\end{equation}
implies \eqref{eq5.6}.
\end{corollary}
\begin{proof} For any $r>u_0$, we have \begin{align*}\int_{r}^{\infty}\frac{\rho^{2\kappa-d+1}d\rho}{\left(\inf_{x\in\R^{d}}\int_{B(0,\rho)}|y|^{2}\nu(x,dy)\right)^{\kappa+1}}&\leq
\frac{1}{S^{\kappa+1}_d}\int_{r}^{\infty}\frac{\rho^{2\kappa-d+1}d\rho}{\left(\inf_{x\in\R^{d}}\int_{u_0}^{\rho}u^{d+1}n(x,u)du\right)^{\kappa+1}}\\
&\leq
\frac{(d+2)^{\kappa+1}}{S^{\kappa+1}_d}\int_{r}^{\infty}\frac{\rho^{2\kappa-d+1}d\rho}{(\rho^{d+2}-u_0^{d+2})^{\kappa+1}\left(\inf_{x\in\R^{d}}n(x,\rho)\right)^{\kappa+1}}\\&\leq
c\int_{r}^{\infty}\frac{\rho^{-d\kappa-2d-1}d\rho}{\left(\inf_{x\in\R^{d}}n(x,\rho)\right)^{\kappa+1}},\end{align*}
where in the second step we employed the fact that $n(x,u)$ is decreasing in $u$ on $(u_0,\infty)$ and $c>(d+2)^{\kappa+1}r^{(d+2)(\kappa+1)}/S^{\kappa+1}_d(r^{d+2}-u_0^{d+2})^{\kappa+1}$ is
arbitrary. Now, the claim is a direct consequence of Proposition
\ref{p5.4}.
\end{proof}
In the following proposition we slightly generalize Theorem \ref{tm5.3}.
\begin{proposition}\label{p5.6}Let $\kappa>0$ and let $\process{F}$  be a
 L\'evy-type process with  symbol
$q(x,\xi)$ and L\'evy triplet $(0,c(x)I,\nu(x,dy))$, satisfying
\begin{align}\label{eq5.11}\liminf_{|\xi|\longrightarrow0}\inf_{x\in\R^{d}}\int_{\R^{d}}\frac{1-\cos\langle\xi,y\rangle}{|\xi|^{2}}\nu(x,dy)>0.\end{align}
 Then,
\eqref{eq5.13}
implies \eqref{eq4.3}.
\end{proposition}
\begin{proof}First, because of the radiality of the function $\xi\longmapsto q(x,\xi)$ for all $x\in\R^{d}$,
\begin{align*}q(x,\xi)-\frac{1}{2}c(x)|\xi|^{2}&=q(x,|\xi|e_i)-\frac{1}{2}c(x)|\xi|^{2}\\
&=\int_{\R^{d}}(1-\cos|\xi|
y_i)\nu(x,dy)\\&\geq\frac{1}{\pi}|\xi|^{2}\int_{B\left(0,\frac{1}{|\xi|}\right)}y_i^{2}\nu(x,dy)\\
&=\frac{1}{d\pi}|\xi|^{2}\int_{B\left(0,\frac{1}{|\xi|}\right)}|y|^{2}\nu(x,dy),\quad \xi\in\R^{d},\end{align*}
where in the third step we used  the fact that $1-\cos u\geq u^{2}/\pi$
for all $|u|\leq1.$ Now, for $r>0$,
\begin{align*}
\int_{B(0,r)}\frac{d\xi}{\left(\inf_{x\in\R^{d}}\left(q(x,\xi)-\frac{1}{2}c(x)|\xi|^{2}\right)\right)^{\kappa+1}}&\leq
d^{\kappa+1}\pi^{\kappa+1}\int_{B(0,r)}\frac{d\xi}{\left(|\xi|^{2}\inf_{x\in\R^{d}}\int_{B\left(0,\frac{1}{|\xi|}\right)}|y|^{2}\nu(x,dy)\right)^{\kappa+1}}\\
&=S_dd^{\kappa+1}\pi^{\kappa+1}\int_{0}^{r}\frac{\rho^{d-1}d\rho}{\left(\rho^{2}\inf_{x\in\R^{d}}\int_{B\left(0,\frac{1}{\rho}\right)}|y|^{2}\nu(x,dy)\right)^{\kappa+1}}
\\&=S_dd^{\kappa+1}\pi^{\kappa+1}\int_{\frac{1}{r}}^{\infty}\frac{\rho^{2\kappa-d+1}d\rho}{\left(\inf_{x\in\R^{d}}\int_{B(0,\rho)}|y|^{2}\nu(x,dy)\right)^{\kappa+1}}.\end{align*}
Finally,
\begin{align*}1&\leq\liminf_{|\xi|\longrightarrow0}\frac{\inf_{x\in\R^{d}}q(x,\xi)}{\inf_{x\in\R^{d}}\left(q(x,\xi)-\frac{1}{2}c(x)|\xi|^{2}\right)}\nonumber\\
&\leq\limsup_{|\xi|\longrightarrow0}\frac{\inf_{x\in\R^{d}}q(x,\xi)}{\inf_{x\in\R^{d}}\left(q(x,\xi)-\frac{1}{2}c(x)|\xi|^{2}\right)}\nonumber\\
&\leq
1+ \frac{\sup_{x\in\R^{d}}c(x)}{\liminf_{|\xi|\longrightarrow0}\inf_{x\in\R^{d}}\int_{\R^{d}}\frac{1-\cos\langle\xi,y\rangle}{|\xi|^{2}}\nu(x,dy)},\end{align*} which together with \eqref{eq5.11} completes the proof.
\end{proof}

Let us now give some applications of the results presented above.

\begin{example}\label{e1}{\rm
 Let  $\process{F}$ be a $d$-dimensional stable-like
                                                         process
with symbol $q(x,\xi)=\gamma(x)|\xi|^{\alpha(x)}.$ The L\'evy measure of   $\process{F}$ is given by $\nu(x,dy)=\delta(x)|y|^{-d-\alpha(x)}dy$, where
$$\delta(x):=\gamma(x)\frac{\alpha(x)
                 2^{\alpha(x)-1}\Gamma((\alpha(x)+1)/2)}{\pi^{1/2}\Gamma(1-\alpha(x)/2)},\quad x\in\R^{d},$$ and $\Gamma(x)$ denotes the Gamma function.
Obviously, $\nu(x,dy)$ satisfies all the assumptions from Theorem \ref{tm5.3}. Thus, by a straightforward application of \eqref{eq5.5} and \eqref{eq5.6} we conclude the same conditions for the $\kappa$-weak and $\kappa$-strong transience of $\process{F}$ obtained in Example \ref{e4.5}.}
\end{example}

\begin{example}\label{e2}{\rm Let $\alpha,\gamma:\R^{d}\longrightarrow(0,\infty)$ be arbitrary bounded and
continuous functions such
that
$\inf_{x\in\R^{d}}\alpha(x)>0$ and $\inf_{x\in\R^{d}}\gamma(x)>0$. Define $n:\R^{d}\times(0,\infty)\longrightarrow(0,\infty)$ by $$n(x,u):=\frac{\gamma(x)}{u^{d+\alpha(x)}}1_{\{v:v\geq 1\}}(u).$$ Because of  continuity of the functions $\alpha(x)$ and $\gamma(x)$, without loss of generality, we can assume that $\int_{\R^{d}}n(x,|y|)dy=1$ for all $x\in\R^{d}.$ Now, there exists an open-set irreducible L\'evy-type process determined by a L\'evy measure and symbol of the form $\nu(x,dy):=n(x,|y|)dy$ and   $q(x,\xi):=\int_{\R^{d}}(1-\cos\langle\xi,y\rangle)\nu(x,dy)$, respectively.
Indeed, let $\chain{F}$ be a $d$-dimensional Markov chain determined by a transition function of the form $p(x,dy):=n(x,|y-x|)$ and let  $\process{N}$
                                                   be a Poisson
                                                   process (with
                                                   parameter
                                                   $1$)
                                                   independent of
                                                   the   chain $\chain{F}$. Then, clearly, the process $F_t:=F_{N_t}$, $t\geq0$, is a strong Markov process with semigroup given by
                                                   $$P_tf(x)=e^{-     t}\sum_{n=0}^{\infty}\frac{
               t^{n}}{n!}\int_{\R^{d}}f(y)p^{n}\left(x,dy\right),\quad  f\in
               B_b(\R^{d}).$$ First, observe that, due to the continuity of the functions $\alpha(x)$ and $\gamma(x)$ and \cite[Proposition 6.1.1]{Meyn-Tweedie-Book-2009} (which states that  $x\longmapsto\int_{\R^{d}}f(y)p(x,dy)$ is a continuous and bounded function for every $f\in C_b(\R^{d})$), $P_t(C_b(\R^{d}))\subseteq C_b(\R^{d})$, $t\geq0$. Here, $C_b(\R^{d})$ denotes the space of  continuous and bounded functions. Next, to see that $P_t(C_\infty(\R^{d}))\subseteq C_\infty(\R^{d})$, $t\geq0$, we proceed as follows. Fix $f\in C_\infty(\R^{d})$ and $\varepsilon>0$. Since
$C_c(\R^{d})$ is dense in $(C_\infty(\R^{d}),||\cdot||_\infty)$, there exists
$f_\varepsilon\in C_c(\R^{d})$ such that
$||f-f_\varepsilon||_\infty<\varepsilon$. Now, we have
\begin{align*}\left|\int_{\R^{d}}f(y) p(x,dy)\right|&\leq\int_{\R^{d}}
|f(y)-f_\varepsilon(y)+f_{\varepsilon}(y)|p(x,dy)\\&\leq\int_{\R^{d}}
|f_\varepsilon(y)|p(x,dy)+\varepsilon\\&\leq||f_\varepsilon||_\infty\int_{{\rm supp}\,f_\varepsilon-x}n(x,|y|)dy
+\varepsilon.\end{align*}
Thus,
due to the boundedness (away from zero and infinity) of the functions $\alpha(x)$ and $\gamma(x)$ and the fact that ${\rm supp}\,f_\epsilon$ is a compact set, we have that the function $x\longmapsto\int_{\R^{d}}f(y)p(x,dy)$ vanishes at infinity. In particular, $\process{P}$ enjoys the Feller property. Finally, obviously $\process{P}$ is also strongly continuous, the domain of the corresponding generator contains $C_\infty(\R^{d})$ and the corresponding  L\'evy measure and symbol are given by $\nu(x,dy)$ and $q(x,\xi)$, respectively.
To check open-set irreducibility of $\process{F}$ is straightforward.

Put $\underline{\alpha}:=\inf_{x\in\R^{d}}\alpha(x)\leq\sup_{x\in\R^{d}}\alpha(x)=:\overline{\alpha}$.
Recall that, due to \cite[Theorems 2.8 and 3.9]{Sandric-TAMS-2014} and \cite[Theorem 4.4]{Sandric-SPA-2015},
\begin{enumerate}
  \item [(i)]if $d\leq2$, then $\process{F}$ is  transient if $\overline{\alpha}<d.$
  \item [(ii)] if $d\geq3$, then $\process{F}$ is always transient.
\end{enumerate}
Next, fix $\kappa>0$.
  Then, by a straightforward application of  Theorem \ref{tm5.3}, we have
 \begin{enumerate}
   \item [(i)] if $\overline{\alpha}<2$, then  $\process{F}$ satisfies \eqref{eq4.1} if
 $\underline{\alpha}(\kappa+1)\geq d$, and  it satisfies \eqref{eq4.3}  if $\overline{\alpha}(\kappa+1)<d$.
   \item [(ii)] if $\overline{\alpha}=\underline{\alpha}=2$, then  $\process{F}$ satisfies \eqref{eq4.1} if
 $2(\kappa+1)> d$, and  it satisfies \eqref{eq4.3} if $2(\kappa+1)\leq d$.
\item [(iii)] if $\underline{\alpha}>2$, then  $\process{F}$ satisfies \eqref{eq4.1} if
 $2(\kappa+1)\geq d$, and  it satisfies \eqref{eq4.3} if $2(\kappa+1)<d$.
\end{enumerate}
}\end{example}
\begin{example}\label{e3}{\rm
Let $\process{L}$ be a $d$-dimensional  L\'evy process with L\'evy measure of the form $\nu(dy)=n(|y|)dy$, where $n:(0,\infty)\longrightarrow(0,\infty)$ is a decreasing (on $(u_0,\infty)$ for some $u_0\geq0$) and regularly varying function with index $\delta\leq-d$ (that is, $\lim_{u\longrightarrow\infty}n(\lambda u)/n(u)=\lambda^{\delta}$ for all $\lambda>0$). First, recall that, according to \cite[Theorems 3.9 and 3.11]{Sandric-TAMS-2014} and \cite[Proposition 4.5]{Sandric-SPA-2015}, $\process{L}$ is transient if
\begin{enumerate}
  \item [(i)] $d=1,2$, and $-2d<\delta\leq-d$ or $\delta=-2d$ and \begin{equation}\label{eq8}\int_{r}\frac{d\rho}{\rho^{2d+1}n(\rho)}<\infty\quad \textrm{for some}\ r>0.\end{equation}
  \item [(ii)] $d\geq3.$
\end{enumerate}
Further, observe that, due to \cite[Theorem 1.5.11]{Bingham-Goldie-Teugels-Book-1987}, for any $-d-2\leq\delta\leq-d$, $$\lim_{\rho\longrightarrow\infty}\frac{\nu(B^{c}_\rho(0))}{\rho^{d}n(\rho)}=\frac{1}{S_d(-d-\delta)}\quad\textrm{and}\quad \lim_{\rho\longrightarrow\infty}\frac{\int_{B_\rho(0)}|y|^{2}\nu(dy)}{\rho^{d+2}n(\rho)}=\frac{1}{S_d(d+\delta+2)}.$$
Consequently,  Proposition \ref{p5.4}, Corollary \ref{c5.5} and \cite[Proposition 1.3.6]{Bingham-Goldie-Teugels-Book-1987} yield that
\begin{itemize}
  \item[(i)]  if  $d\geq3$ and $\delta<-d-2$, then $\process{L}$ is $\kappa$-weakly transient if, and only if, $2(\kappa+1)\geq d$.
 \item[(ii)] if $d=1,2$, $\delta=-2d$ and \eqref{eq8} holds, then $\process{L}$ is $\kappa$-weakly transient if, and only if, $2(\kappa+1)>d$.
\item[(iii)] if $d\geq3$ and $\delta=-d-2$, then $\process{L}$ is $\kappa$-weakly transient if, and only if, $2(\kappa+1)>d$.
 \item [(iv)] if  $d=1$ and $-2<\delta<-1$, then $\process{L}$ is $\kappa$-weakly transient if, and only if, $\kappa+2+\delta(\kappa+1)\leq0$.
  \item [(v)] if  $d\geq2$ and $-d-2<\delta<-d$, then $\process{L}$ is $\kappa$-weakly transient if, and only if, $d(\kappa+2)+\delta(\kappa+1)\leq0$.
   \item[(vi)] if $\delta=-d$, then $\process{L}$ is always $\kappa$-strongly transient.
\end{itemize}

}\end{example}

Finally, as a consequence of Theorem \ref{tm5.3} we can conclude the following comparison conditions for the $\kappa$-weak and $\kappa$-strong transience.
\begin{theorem}  \label{tm5.7}
Let $\kappa>0$ and let $\process{F}$ and  $\process{\bar{F}}$ be
L\'evy-type processes with  symbols $q(x,\xi)$ and
$\bar{q}(x,\xi)$ and  L\'evy measures $\nu(x,dy)$ and
$\bar{\nu}(x,dy)$, respectively. Assume  that there exists
$u_0\geq0$ such that:
\begin{enumerate}
\item [(i)] $\nu(x,dy)=n(x,|y|)dy$ on $\mathcal{B}(B^{c}(0,u_0))$ for some Borel function $n:\R^{d}\times(0,\infty)\longrightarrow(0,\infty)$ which is decreasing on $(u_0,\infty)$ for all $x\in\R^{d}$;
\item [(ii)] $\nu\left(x,B^{c}(0,u)\right)\geq\bar{\nu}\left(x,B^{c}(0,u)\right)$ for all  $x\in\R^{d}$ and all $u>u_0$.
 \end{enumerate}   Then,
$$\int_{B(0,r)}\frac{d\xi}{\left(\sup_{x\in\R^{d}}q(x,\xi)\right)^{\kappa+1}}=\infty\quad \textrm{for all}\ r>0$$
implies
$$\int_{B(0,r)}\frac{d\xi}{\left(\sup_{x\in\R^{d}}\bar{q}(x,\xi)\right)^{\kappa+1}}=\infty \quad \textrm{for all}\ r>0.$$
In addition, if $q(x,\xi)$ satisfies \eqref{eq5.2}, then
$$\int_{B(0,r)}\frac{d\xi}{\left(\inf_{x\in\R^{d}}\bar{q}(x,\xi)\right)^{\kappa+1}}<\infty\quad \textrm{for some}\ r>0$$
implies
$$\int_{B(0,r)}\frac{d\xi}{\left(\inf_{x\in\R^{d}}q(x,\xi)\right)^{\kappa+1}}<\infty\quad \textrm{for some}\ r>0.$$
\end{theorem}

 \section*{Acknowledgement}
 This work was supported by the Croatian Science Foundation under Grant 3526 and NEWFELPRO Programme  under Grant 31.
The author thanks the referee for   helpful comments
and careful corrections.

\bibliographystyle{alpha}
\bibliography{References}

\end{document}